\documentclass{mcom-l}
\usepackage{amsmath,url,amssymb,amscd}
\usepackage{booktabs}
\usepackage{tikz}

\copyrightinfo{}{}

\newcommand{\afrac}[2]{\genfrac{}{}{0pt}{}{#1}{#2}}
\newcommand{\F}{\mathbf{F}}
\newcommand{\Fp}{\mathbf{F}_p}
\newcommand{\barFp}{\overline{\F}_p}

\newcommand{\CC}{\mathbf{C}}
\newcommand{\C}{\mathbf{C}}
\newcommand{\HH}{\mathbf{H}}
\newcommand{\Q}{\mathbf{Q}}
\newcommand{\ZZ}{\mathbf{Z}}
\newcommand{\Z}{\mathbf{Z}}
\newcommand{\R}{\mathbf{R}}
\newcommand{\M}{\textsf{M}}
\newcommand{\Zm}{\Z/m\Z}
\newcommand{\mf}{\mathfrak}
\newcommand{\Wf}{\mathfrak{f}}
\renewcommand{\O}{\mathcal{O}}
\newcommand{\gp}{\mathfrak p}
\newcommand{\isar}{\ \smash{\mathop{\longrightarrow}\limits^{\thicksim}}\ }
\def\mapright#1{\ \smash{\mathop{\longrightarrow}\limits^{#1}}\ }

\def\Gal{\operatorname{Gal}}

\def\End{\operatorname{End}}
\def\Pic{\operatorname{cl}}

\def\tr{\operatorname{tr}}
\def\Li{\operatorname{Li}}
\def\SL{\operatorname{SL}}
\def\GL{\operatorname{GL}}

\def\disc{\operatorname{disc}}

\def\Ell{\text{\rm Ell}}
\def\mfl{{\mathfrak{l}}}
\newcommand{\kron}[2]{\left(\frac{#1}{#2}\right)}
\newcommand{\inkron}[2]{\genfrac {(}{)}{0.9pt}{}{#1}{#2}}
\renewcommand{\vec}[1]{{\boldsymbol{#1}}}
\newcommand{\idx}[2]{[#1\hspace{1.5pt}\text{\rm :}\hspace{2pt}#2]}
\newcommand{\eps}{\varepsilon}
\newcommand{\invexp}{\text{-1}}

\newcommand{\algstart}[2]{\smallskip\noindent{\bf Algorithm} #1. \emph{#2}\begin{enumerate}}
\newcommand{\algend}{\end{enumerate}\vspace{4pt}}
\newcommand{\algitem}{\vspace{2pt}\item}

\newtheorem{theorem}{Theorem}[section]
\newtheorem*{thm1}{Theorem~1} 

\newtheorem{lemma}[theorem]{Lemma}

\newtheorem{definition}[theorem]{Definition}
\newtheorem{example}[theorem]{Example}
\theoremstyle{definition}
\newtheorem{algorithm}[theorem]{Algorithm}

\renewcommand\labelenumi{\theenumi.}
\renewcommand\labelenumii{\theenumii.}

\begin{document}

\title[Modular polynomials via isogeny volcanoes]{Modular polynomials via isogeny volcanoes}

\author{Reinier Br\"oker, Kristin Lauter, and Andrew V. Sutherland}
\address{Brown University, Box 1917, 151 Thayer Street, Providence, Rhode Island  02192}
\email{reinier@math.brown.edu}
\address{Microsoft Research, One Microsoft Way, Redmond, Washington 98052}
\email{klauter@microsoft.com}
\address{Massachusetts Institute of Technology, Cambridge, Massachusetts 02139}
\email{drew@math.mit.edu}

\subjclass[2000]{Primary 11Y16 ; Secondary  11G15, 11G20, 14H52}
\date{}
\dedicatory{}

\begin{abstract}
We present a new algorithm to compute the classical modular polynomial 
$\Phi_l$ in the rings $\ZZ[X,Y]$ and $(\Zm)[X,Y]$, for a prime $l$ and any positive integer $m$.
Our approach uses the graph of $l$-isogenies to efficiently compute $\Phi_l\bmod p$ 
for many primes $p$ of a suitable form, and then applies the Chinese Remainder Theorem (CRT).
Under the Generalized Riemann Hypothesis (GRH), we achieve an
expected running time of $O(l^3 (\log l)^3\log\log l)$,
and compute $\Phi_l\bmod m$ using $O(l^2(\log l)^2+ l^2\log m)$ space.
We have used the new algorithm to compute $\Phi_l$ with $l$ over 5000,
and $\Phi_l\bmod m$ with $l$ over 20000.
We also consider several modular functions $g$ for which $\Phi_l^g$ is smaller
than $\Phi_l$, allowing us to handle $l$ over 60000. 
\end{abstract}

\maketitle
\section{Introduction}
For a prime~$l$, the classical modular polynomial $\Phi_l$ is the minimal 
polynomial of the function $j(lz)$ over the field $\CC(j)$, where $j(z)$ is 
the modular $j$-function. The polynomial~$\Phi_l$ parametrizes elliptic 
curves $E$ together with an isogeny $E \rightarrow E'$ of degree~$l$. From
classical results, we know that $\Phi_l$ lies in the ring $\ZZ[X,Y]$ and satisfies
$\Phi_l(X,Y) = \Phi_l(Y,X)$, with degree $l+1$ in both variables \cite[\S69]{Weber:Algebra}.

The fact that the moduli interpretation of $\Phi_l$ remains valid modulo
primes $p \not = l$ was crucial to the improvements made by Atkin and Elkies
to Schoof's point-counting algorithm \cite{Elkies:AtkinBirthday,Schoof:ECPointCounting2}.  More
recently, the polynomials $\Phi_l\bmod p$ have been used to compute Hilbert class polynomials~\cite{Belding:HilbertClassPolynomial,Sutherland:HilbertClassPolynomials}, and to determine the endomorphism ring of an elliptic curve over a finite field~\cite{BissonSutherland:Endomorphism}.  Explicitly computing $\Phi_l$ is notoriously difficult, primarily due to its large size.  As shown in~\cite{CohenPaula:ModularPolynomials}, the logarithmic height of its largest coefficient is $6l\log l+O(l)$, thus its total size is 
\begin{equation}\label{spacebound}
O(l^3 \log l).
\end{equation}
As this bound suggests, the size of $\Phi_l$ grows quite rapidly; the binary representation of $\Phi_{79}$ already exceeds one megabyte, and $\Phi_{659}$ is larger than a gigabyte.

The polynomial $\Phi_l$ can be computed by comparing coefficients in the
Fourier expansions of $j(z)$ and $j(lz)$, an approach considered by several
authors~\cite{Blake:ModularPolynomials, Elkies:AtkinBirthday,
Herrmann:FourierCoefficients, Ito:ModularEquation, Kaltofen:ModularEquation,
LMMS:PointCounting, Morain:PointCounting}.  As detailed in
\cite{Blake:ModularPolynomials}, this only requires integer arithmetic, and may
be performed modulo~$p$ for any prime $p > 2l+2$.  The time to compute $\Phi_l
\bmod p$ is then $O(l^{3+\eps}(\log p)^{1+\varepsilon})$, and for a sufficiently large
$p$ this yields an $O(l^{4+\varepsilon})$ time algorithm to compute $\Phi_l$ over
$\ZZ$.  Alternatively (and preferably), one computes $\Phi_l$ modulo several
smaller primes and applies the Chinese Remainder Theorem, as suggested in
\cite{Blake:ModularPolynomials,Herrmann:FourierCoefficients,
LMMS:PointCounting, Morain:PointCounting}.

An alternative CRT-based approach appears in~\cite{CharlesLauter:ModPoly}.
This algorithm uses isogenies between supersingular elliptic curves defined
over a finite field, and computes $\Phi_l\bmod p$ in time $O(l^{4+\eps}(\log
p)^{2+\eps} + (\log p)^{4+\eps}$), under the GRH.

In~\cite{Enge:ModularPolynomials}, Enge uses interpolation and fast
floating-point evaluations to compute $\Phi_l\in\ZZ[X,Y]$ in time $O(l^3(\log
l)^{4+\eps})$, under reasonable heuristic assumptions.  The complexity of this
method is nearly optimal, quasi-linear in the size of $\Phi_l$.  However, most
applications actually use $\Phi_l$ in a finite field $\mathbf{F}_{p^n}$, and
$\Phi_l\bmod p$ may be much smaller than $\Phi_l$.  In general,  Enge's
algorithm can compute $\Phi_l$ and reduce it modulo~$p$ much faster than either
of the methods above can compute $\Phi_l\bmod p$, but this may use an excessive
amount of space.  For large $l$ this approach becomes impractical, even when
$\Phi_l\bmod p$ is reasonably small.

Here we present a new method to compute $\Phi_l$, either over the integers or
modulo an arbitrary positive integer~$m$, including $m \le l$.  Our algorithm is both
asymptotically and practically faster than alternative methods, and achieves
essentially optimal space complexity.  More precisely, we prove the following
result.

\begin{thm1}\label{main-thm}
Let $l$ denote an odd prime and $m$ a positive integer.
Algorithm~\ref{alg2} correctly computes $\Phi_l\in(\Zm)[X,Y]$.
Under the GRH, it runs in expected time
$$
O(l^3 \log ^3 l\log\log l),
$$
using $O(l^2\log lm)$ expected space.
\end{thm1}

To compute $\Phi_l$ over $\Z$, we choose a modulus $m$ that is large enough to uniquely
determine the coefficients, via an explicit height bound proven in
\cite{BrokerSutherland:PhiHeightBound}.  In general, we may assume $\log m=O(l\log l)$,
since otherwise $\Phi_l$ and $\Phi_l\bmod m$ are effectively the same (hence the
time bound does not depend on $m$).

Our algorithm is of the \emph{Las
Vegas} type, a probabilistic algorithm whose output is unconditionally correct;
the GRH is only used to analyze its running time.  We have used it to compute
$\Phi_l$ for all $l<3600$, and many larger $l$ up to $5003$.  The largest
previous computation of which we are aware has $l=1009$.  Working
modulo $m$ we can go further; we have computed $\Phi_l$ modulo a 256-bit
integer $m$ with $l=20011$.

Applications that rely on $\Phi_l$ can often improve their running times by
using alternative modular polynomials that have smaller coefficients.  Our
algorithm can be adapted to compute polynomials $\Phi_l^g$ relating $g(z)$ and
$g(lz)$, for modular functions~$g$ that share certain properties with~$j$.  This
includes the cube root $\gamma_2$ of $j$, and we are then able to compute
$\Phi_l\bmod m$ more quickly by reconstructing it from $\Phi_l^{\gamma_2}\bmod m$,
capitalizing on a suggestion in \cite{Elkies:AtkinBirthday}.  Other examples
include simple and double eta-quotients, the Atkin functions, and the
Weber $\Wf$-function. The last is especially
attractive, since the modular polynomials for $\Wf$ are approximately 1728 times
smaller than those for $j$.  This has allowed us to compute modular polynomials
$\Phi_l^{\Wf}$ with $l$ as large as 60013. 

The outline of this article is as follows. In Section~2 we give a rough 
overview of our new algorithm. The theory behind the algorithm is presented
in Sections~3--5. We present the algorithm, prove its correctness and 
analyze its runtime in Section~6. Section~7 deals with modular polynomials
for modular functions other than $j$, and a final Section~8 contains computational
results.

\section{Overview}
\noindent
Our basic strategy is a standard CRT approach: we compute $\Phi_l\bmod p$ for various primes~$p$ and use the Chinese Remainder Theorem to 
recover $\Phi_l\in\Z[X,Y]$.
Alternatively, the explicit CRT (mod $m$) allows us to directly compute $\Phi_l\in(\Zm)[X,Y]$, via~\cite[Thm.~3.1]{Bernstein:ModularExponentiation}.
By applying the algorithm of~\cite[\S6]{Sutherland:HilbertClassPolynomials}, this can be accomplished in $O(l^2\log lm)$ space, even though the total size of all the $\Phi_l\bmod p$ is $O(l^3\log l)$.

Our method for computing $\Phi_l\bmod p$ is new, and applies only to certain primes~$p$.  
Strategic prime selection has been used effectively in other CRT-based 
algorithms, such as~\cite{Sutherland:HilbertClassPolynomials}, and it is especially helpful here.
Working in the finite field~$\Fp$, we select $l+2$ distinct values $j_i$, compute $\Phi_l(X,j_i)\in\Fp[X]$ for each, and then interpolate the coefficients of $\Phi_l\in (\Fp[Y])[X]$ as polynomials in $\Fp[Y]$.
The key lies in our choice of $p$, which allows us to select particular interpolation points that greatly facilitate the computation.
We are then able to compute $\Phi_l\bmod p$ in expected time
\begin{equation}\label{modpbound}
O(l^2(\log p)^3\log\log p).
\end{equation}
In contrast to the methods above, this is quasi-linear in the size of $\Phi_l\bmod p$.

Our algorithm exploits the structure of the $l$-isogeny graph $G_l$ defined on the set of $j$-invariants of elliptic curves over $\Fp$.
Each edge in this graph corresponds to an $l$-isogeny; the edge $(j_1,j_2)$ is present if and only if $\Phi_l(j_1,j_2)=0$.
As described in~\cite{Fouquet:IsogenyVolcanoes,Kohel:thesis}, the ordinary components of this graph have a particular structure known as an $l$-\emph{volcano}.
Depicted in Figure 1 are a set of four $l$-volcanoes, each with two levels: the \emph{surface} (at the top), and the \emph{floor} (on the bottom).
Note that each vertex $j_i$ on the surface has $l+1$ neighbors, these are the roots of $\Phi_l(X,j_i)\in\Fp[X]$, and there are at least $l+2$ such $j_i$.
\smallskip

\begin{figure}[htp]
\begin{tikzpicture}
\draw (-4.5,0) ellipse (1 and 0.1);
\draw (-1.5,0) ellipse (1 and 0.1);
\draw (1.5,0) ellipse (1 and 0.1);
\draw (4.5,0) ellipse (1 and 0.1);
\draw (-4.5,0.1) -- (-4.6,-0.06);
\draw (-4.5,0.1) -- (-4.56,-0.06);
\draw (-4.5,0.1) -- (-4.52,-0.06);
\draw (-4.5,0.1) -- (-4.48,-0.06);
\draw (-4.5,0.1) -- (-4.44,-0.06);
\draw (-4.5,0.1) -- (-4.4,-0.06);
\draw[fill=red] (-4.5,0.1) circle (0.04);
\draw (-1.5,0.1) -- (-1.6,-0.05);
\draw (-1.5,0.1) -- (-1.56,-0.05);
\draw (-1.5,0.1) -- (-1.52,-0.05);
\draw (-1.5,0.1) -- (-1.48,-0.05);
\draw (-1.5,0.1) -- (-1.44,-0.05);
\draw (-1.5,0.1) -- (-1.4,-0.05);
\draw[fill=red] (-1.5,0.1) circle (0.04);
\draw (1.5,0.1) -- (1.6,-0.05);
\draw (1.5,0.1) -- (1.56,-0.05);
\draw (1.5,0.1) -- (1.52,-0.05);
\draw (1.5,0.1) -- (1.48,-0.05);
\draw (1.5,0.1) -- (1.44,-0.05);
\draw (1.5,0.1) -- (1.4,-0.05);
\draw[fill=red] (1.5,0.1) circle (0.04);
\draw (4.5,0.1) -- (4.60,-0.06);
\draw (4.5,0.1) -- (4.56,-0.06);
\draw (4.5,0.1) -- (4.52,-0.06);
\draw (4.5,0.1) -- (4.48,-0.06);
\draw (4.5,0.1) -- (4.44,-0.06);
\draw (4.5,0.1) -- (4.4,-0.06);
\draw[fill=red] (4.5,0.1) circle (0.04);
\draw (-5,-0.1) -- (-5.35,-0.7);
\draw[fill=red] (-5.35,-0.7) circle (0.04);
\draw (-5,-0.1) -- (-5.21,-0.7);
\draw[fill=red] (-5.21,-0.7) circle (0.04);
\draw (-5,-0.1) -- (-5.07,-0.7);
\draw[fill=red] (-5.07,-.7) circle (0.04);
\draw (-5,-0.1) -- (-4.93,-0.7);
\draw[fill=red] (-4.93,-0.7) circle (0.04);
\draw (-5,-0.1) -- (-4.79,-0.7);
\draw[fill=red] (-4.79,-0.7) circle (0.04);
\draw (-5,-0.1) -- (-4.65,-0.7);
\draw[fill=red] (-4.65,-0.7) circle (0.04);
\draw[fill=red] (-5,-0.1) circle (0.04);

\draw (-4,-0.1) -- (-4.35,-0.7);
\draw[fill=red] (-4.35,-0.7) circle (0.04);
\draw (-4,-0.1) -- (-4.21,-0.7);
\draw[fill=red] (-4.21,-0.7) circle (0.04);
\draw (-4,-0.1) -- (-4.07,-0.7);
\draw[fill=red] (-4.07,-0.7) circle (0.04);
\draw (-4,-0.1) -- (-3.93,-0.7);
\draw[fill=red] (-3.93,-0.7) circle (0.04);
\draw (-4,-0.1) -- (-3.79,-0.7);
\draw[fill=red] (-3.79,-0.7) circle (0.04);
\draw (-4,-0.1) -- (-3.65,-0.7);
\draw[fill=red] (-3.65,-0.7) circle (0.04);
\draw[fill=red] (-4,-0.1) circle (0.04);

\draw (-2,-0.1) -- (-2.35,-0.7);
\draw[fill=red] (-2.35,-0.7) circle (0.04);
\draw (-2,-0.1) -- (-2.21,-0.7);
\draw[fill=red] (-2.21,-0.7) circle (0.04);
\draw (-2,-0.1) -- (-2.07,-0.7);
\draw[fill=red] (-2.07,-0.7) circle (0.04);
\draw (-2,-0.1) -- (-1.93,-0.7);
\draw[fill=red] (-1.93,-0.7) circle (0.04);
\draw (-2,-0.1) -- (-1.79,-0.7);
\draw[fill=red] (-1.79,-0.7) circle (0.04);
\draw (-2,-0.1) -- (-1.65,-0.7);
\draw[fill=red] (-1.65,-0.7) circle (0.04);
\draw[fill=red] (-2,-0.1) circle (0.04);

\draw (-1,-0.1) -- (-1.35,-0.7);
\draw[fill=red] (-1.35,-0.7) circle (0.04);
\draw (-1,-0.1) -- (-1.21,-0.7);
\draw[fill=red] (-1.21,-0.7) circle (0.04);
\draw (-1,-0.1) -- (-1.07,-0.7);
\draw[fill=red] (-1.07,-0.7) circle (0.04);
\draw (-1,-0.1) -- (-0.93,-0.7);
\draw[fill=red] (-0.93,-0.7) circle (0.04);
\draw (-1,-0.1) -- (-0.79,-0.7);
\draw[fill=red] (-0.79,-0.7) circle (0.04);
\draw (-1,-0.1) -- (-0.65,-0.7);
\draw[fill=red] (-0.65,-0.7) circle (0.04);
\draw[fill=red] (-1,-0.1) circle (0.04);

\draw (1,-0.1) -- (1.35,-0.7);
\draw[fill=red] (1.35,-0.7) circle (0.04);
\draw (1,-0.1) -- (1.21,-0.7);
\draw[fill=red] (1.21,-0.7) circle (0.04);
\draw (1,-0.1) -- (1.07,-0.7);
\draw[fill=red] (1.07,-0.7) circle (0.04);
\draw (1,-0.1) -- (0.93,-0.7);
\draw[fill=red] (0.93,-0.7) circle (0.04);
\draw (1,-0.1) -- (0.79,-0.7);
\draw[fill=red] (0.79,-0.7) circle (0.04);
\draw (1,-0.1) -- (0.65,-0.7);
\draw[fill=red] (0.65,-0.7) circle (0.04);
\draw[fill=red] (1,-0.1) circle (0.04);

\draw (2,-0.1) -- (2.35,-0.7);
\draw[fill=red] (2.35,-0.7) circle (0.04);
\draw (2,-0.1) -- (2.21,-0.7);
\draw[fill=red] (2.21,-0.7) circle (0.04);
\draw (2,-0.1) -- (2.07,-0.7);
\draw[fill=red] (2.07,-0.7) circle (0.04);
\draw (2,-0.1) -- (1.93,-0.7);
\draw[fill=red] (1.93,-0.7) circle (0.04);
\draw (2,-0.1) -- (1.79,-0.7);
\draw[fill=red] (1.79,-0.7) circle (0.04);
\draw (2,-0.1) -- (1.65,-0.7);
\draw[fill=red] (1.65,-0.7) circle (0.04);
\draw[fill=red] (2,-0.1) circle (0.04);

\draw (4,-0.1) -- (4.35,-0.7);
\draw[fill=red] (4.35,-0.7) circle (0.04);
\draw (4,-0.1) -- (4.21,-0.7);
\draw[fill=red] (4.21,-0.7) circle (0.04);
\draw (4,-0.1) -- (4.07,-0.7);
\draw[fill=red] (4.07,-0.7) circle (0.04);
\draw (4,-0.1) -- (3.93,-0.7);
\draw[fill=red] (3.93,-0.7) circle (0.04);
\draw (4,-0.1) -- (3.79,-0.7);
\draw[fill=red] (3.79,-0.7) circle (0.04);
\draw (4,-0.1) -- (3.65,-0.7);
\draw[fill=red] (3.65,-0.7) circle (0.04);
\draw[fill=red] (4,-0.1) circle (0.04);

\draw (5,-0.1) -- (5.35,-0.7);
\draw[fill=red] (5.35,-0.7) circle (0.04);
\draw (5,-0.1) -- (5.21,-0.7);
\draw[fill=red] (5.21,-0.7) circle (0.04);
\draw (5,-0.1) -- (5.07,-0.7);
\draw[fill=red] (5.07,-0.7) circle (0.04);
\draw (5,-0.1) -- (4.93,-0.7);
\draw[fill=red] (4.93,-0.7) circle (0.04);
\draw (5,-0.1) -- (4.79,-0.7);
\draw[fill=red] (4.79,-0.7) circle (0.04);
\draw (5,-0.1) -- (4.65,-0.7);
\draw[fill=red] (4.65,-0.7) circle (0.04);
\draw[fill=red] (5,-0.1) circle (0.04);
\end{tikzpicture}
\vspace{12pt}

\begin{minipage}{0.75\linewidth}
\textsc{figure} 1.  A set of $l$-volcanoes arising from Theorem~\ref{theprimes}.  In this example $l=7$ splits into ideals of order 3 in $\Pic(\O)$ and we have $h(\O)=12$ surface curves and $h(R)=72$ floor curves.
\end{minipage}
\end{figure}
\smallskip

This configuration contains enough information to compute the $l+2$ polynomials $\Phi_l(X,j_i)$ that we need to interpolate $\Phi_l(X,Y)\bmod p$.  It is not an arrangement that is likely to arise by chance; it is achieved by our choice of the order $\O$ and the primes~$p$ that we use.
To further simplify our task, we choose $p$ so that vertices on the surface correspond to curves with $\Fp$-rational $l$-torsion.
Our ability to obtain such primes is guaranteed by Theorems~\ref{theprimes} and~\ref{smallprimes}, proven in Section~\ref{ExplicitCMTheory}.

The curves on the surface all have the same endomorphism ring type, isomorphic to an imaginary quadratic order $\O$.
Their $j$-invariants are precisely the roots of the Hilbert class polynomial $H_\O\in\Z[X]$.
As described in~\cite{Belding:HilbertClassPolynomial}, the roots of $H_\O$ may be enumerated via the action of the ideal class group $\Pic(\O)$.  To do so efficiently, we use an algorithm of~\cite{Sutherland:HilbertClassPolynomials} to compute a \emph{polycyclic presentation} for $\Pic(\O)$ that allows us to enumerate the roots of $H_\O$ via isogenies of low degree, typically much smaller than~$l$.
We may use this presentation to determine the action of any element of $\Pic(O)$, including those that act via $l$-isogenies.
This allows us to identify the $l$-isogeny cycles that form the surfaces of the volcanoes in Figure 1.

Similarly, the vertices on the floor are the roots of $H_R$, where $R$ is the order of index $l$ in $\O$, and we use a polycyclic presentation of $\Pic(R)$ to enumerate them.
To identity children of a common parent (siblings), we exploit the fact that siblings lie in a cycle of $l^2$-isogenies, which we identify using our presentation of $\Pic(R)$.
It remains only to connect each parent to one of its children.  This may be achieved by using V\'{e}lu's formula~\cite{Velu:Isogenies} to compute an $l$-isogeny from the surface to the floor.
By matching each parent to a group of siblings, we avoid the need to compute an $l$-isogeny to every child, which is critical to obtaining the complexity bound in (\ref{modpbound}).

Below is a simplified version of the algorithm to compute $\Phi_l\bmod p$.
\begin{algorithm}\label{alg1}
{Let $l$ be an odd prime, and let $\O$ be an imaginary quadratic order of discriminant $D$ with class number $h(\O)\ge l+2$.
Let $p\equiv 1\bmod l$ be a prime satisfying $4p=t^2-l^2v^2D$ for some integers~$t$ and~$v$ with $l\nmid v$.
Let $R$ be the order of index $l$ in $\O$.
Compute $\Phi_l\bmod p$ as follows:}
\begin{enumerate}
\algitem
Find a root of $H_\O$ over $\Fp$.
\algitem
Enumerate the roots $j_i$ of $H_\O$ and identify the $l$-isogeny cycles.
\algitem
For each $j_i$ find an $l$-isogenous $j$ on the floor.
\algitem
Enumerate the roots of $H_R$ and identify the $l^2$-isogeny cycles.
\algitem
For each $j_i$ compute $\Phi_l(X,j_i)=\prod_{(j_i,j_k)\in G_l}(X-j_k)$.
\algitem
Interpolate $\Phi_l\in (\Fp[Y])[X]$ using the $j_i$ and the polynomials $\Phi_l(X,j_i)$.
\end{enumerate}
\end{algorithm}
The conditions on the inputs $l$, $\O$, and $p$ suffice to ensure that Theorem~\ref{theprimes} is satisfied, so that we have a configuration of $l$-volcanoes similar to the example in Figure 1.
We use the same $\O$ for each $p$, so the Hilbert class polynomial $H_\O$ may be precomputed, but we do not need to compute $H_R$, instead we enumerate its roots by applying the Galois action of $\Pic(R)$ to a root obtained in Step 3.

A more detailed version of Algorithm~\ref{alg1} appears in Section~6 together with Algorithm~\ref{alg2}, which selects the order $\O$ and the primes $p$, and performs the CRT computations needed to determine $\Phi_l$ over $\Z$, or modulo~$m$.

\section{Orders in imaginary quadratic fields}
It is a classical fact that the endomorphism ring of an ordinary elliptic
curve over a finite field is isomorphic to an imaginary quadratic order~$\O$.
The order $\O$ is necessarily contained in the maximal order $\O_K$ of its
fraction field~$K$, but we quite often have $\O\subsetneq\O_K$.  As 
most textbooks on algebraic number theory focus on
maximal orders, we first develop some useful tools for working with non-maximal
orders.  To simplify the presentation, we work throughout with fields of
discriminant $d_K<-4$, ensuring that we always have the unit groups
$\O^*=\O_K^*=\{\pm1\}$.  We use $\inkron{d_K}{p}$ to denote the Kronecker symbol,
which is~$-1, 0$, or $1$ as the prime $p$ splits, ramifies, or remains inert in $K$ (respectively).
\smallskip

Let $\O$ be a (not necessarily maximal) order in a quadratic field $K$ of discriminant $d_K<-4$.
Let $N$ be a positive integer prime to the conductor $u=\idx{\O_K}{\O}$.
The order $R = \Z + N\O$ has index~$N$ in $\O$, and its ideal class group $\Pic(R)$ is an extension of $\Pic(\O)$.
More precisely, as in~\cite[Thm.~6.7]{Stevenhagen:NumberRings}, there is an exact sequence
\begin{equation}\label{exactsequence}
1 \mapright{} (\O/N\O)^* / (\Z/N\Z)^* \mapright{} \Pic(R) 
\mapright{\varphi} \Pic(\O) \mapright{} 1,
\end{equation}
where $\varphi$ maps the class $[I]$ to the class $[I\O]$.
For $R$-ideals prime to~$uN$, the underlying map $I\mapsto I\O$ preserves 
`norms', that is, $\idx{R}{I}=\idx{\O}{I\O}$, as in~\cite[Prop.~7.20]{Cox:ComplexMultiplication}.
We have a particular interest in the kernel of the map $\varphi$.

\begin{lemma}\label{cyclic} In the exact sequence above, if $N=p^n$ is a power of an unramified odd prime $p$, then $\ker\varphi$ is cyclic of order $p^{n-1}\bigl(p-\inkron{d_K}{p}\bigr)$.
\end{lemma}
\begin{proof}
We compute the structure of $\ker\varphi\cong (\O/p^n\O)^*/(\Z/p^n\Z)^*$.
The group $(\Z/p^n\Z)^*$ is cyclic, isomorphic to the \emph{additive} group $(\Z/(p-1)\Z)\times(\Z/p^{n-1}\Z)$.  We now apply~\cite[Cor.~4.2.11]{Cohen:CANT2} to compute the structure of $(\O/p^n\O)^*$:

\begin{equation}\label{unitgroup}
(\O/p^n\O)^*\cong
\begin{cases}
(\Z/(p-1)\Z)^2\times(\Z/p^{n-1}\Z)^2, &\text{if $p$ splits in $K$;}\\
(\Z/(p^2-1)\Z)\times(\Z/p^{n-1}\Z)^2, &\text{if $p$ is inert in $K$.}
\end{cases}
\end{equation}
In both cases, the factor $(\Z/p^{n-1}\Z)$ of $(\Z/p^n\Z)^*$ is a maximal cyclic subgroup of the Sylow $p$-subgroup of $(\O/p^n\O)^*$, and must correspond to a direct summand.
Thus the $p$-rank of the quotient $(\O/p^n\O)^*/(\Z/p^n\Z)^*$ is 1.
The order of the factor $(\Z/(p-1)\Z)$ of $(\Z/p^n\Z)^*$ is not divisible by $p$, and must correspond to a subgroup of a cyclic factor of $(\O/p^n\O)^*$ in both cases.  It follows that the quotient is cyclic.
The calculations above also show that $\#(\O/p^n\O)^*/(\Z/p^n\Z)^* = p^{n-1}\bigl(p-\inkron{d_K}{p}\bigr)$.
\end{proof}
\noindent
Even when $\ker\varphi$ is not necessarily cyclic, the size of $\ker \varphi$ is as
in Lemma~\ref{cyclic}.  More generally, the exact sequence (\ref{exactsequence})
can be used to derive the formula
\begin{equation}\label{classnumber}
h(\O)=h(\O_K)u\prod_{p|u}\left(1-\kron{d_K}{p}p^{-1}\right),
\end{equation}
as in~\cite[Thm.~7.24]{Cox:ComplexMultiplication}.

We now describe a particular representation of $\ker \varphi$ when $N=l$ is 
prime. In this case $\ker\varphi$ is cyclic, of order $l-\inkron{d_K}{l}$;
this follows from Lemma~\ref{cyclic} for $l>2$, and from (\ref{classnumber}) for $l=2$.
Let $\O=\Z[\tau]$ for some $\tau\in K$ that is coprime to~$l$.
There are exactly $l+1$ index~$l$ sublattices of~$\O$: the order $R$, and lattices $S_i=l\Z+(\tau+i)\Z$, for $i$ from 0 to $l-1$.
Each $\O$-ideal of norm $l$ corresponds to one of the $S_i$.
The remaining $S_i$ are fractional invertible $R$-ideals corresponding to proper $R$-ideals
\begin{equation}\label{Ji}
J_i=lS_i=l^2\Z+l(\tau+i)\Z,
\end{equation}
for which $R=\{\beta\in K:\beta J_i\subset J_i\}$.
Exactly $1+\inkron{d_K}{l}$ of the $S_i$ are $\O$-ideals, leaving $l-1-\inkron{d_K}{l}$ proper $R$-ideals~$J_i$.
These are all non-principal and inequivalent in $\Pic(R)$, and each lies in $\ker \varphi$, since we have $J_i\O=l\O$. The invertible $J_i$ are exactly the non-trivial elements of $\ker \varphi$.
We summarize with the following lemma, which guarantees that we can find a generator for the cyclic group $\ker \varphi$ that has norm~$l^2$.

\begin{lemma}\label{generators}
If $N=l$ is prime in the exact sequence $(\ref{exactsequence})$, then the $R$-ideal $lR$ and the invertible $R$-ideals $J_i$ defined in $(\ref{Ji})$ are representatives for $\ker \varphi$.
In particular, $\ker \varphi$ is generated by the class of an invertible $R$-ideal with norm $l^2$.
\end{lemma}
\noindent
This representation of $\ker\varphi$ has proven useful in other 
settings~\cite{Castagnos:NICECryptanalysis}. We use it to obtain 
the $l^2$-isogeny cycles we need in Step 4 of Algorithm~\ref{alg1}.

We conclude this section with a theorem that allows us to construct arbitrarily large class groups that are generated by elements of bounded norm.

\begin{theorem}\label{theorders}
Let $\O$ be an order in a quadratic field of discriminant $d_K<-4$, and 
let $p\nmid\disc(\O)$ be an odd prime. Let $\mathcal{P}$ be a set of primes 
that do not divide $p\idx{\O_K}{\O}$.
For $n\in\Z_{\ge 0}$, let $R_n$ denote the order $\Z+p^n\O$, and let $G_n$ be the subgroup of $\Pic(R_n)$ generated by the set $S_n$ of classes of $R_n$-ideals with norms in $\mathcal{P}$.

Then if $G_2=\Pic(R_2)$, we have $G_n=\Pic(R_n)$ for every $n\in\Z_{\ge 0}$.
\end{theorem}
\begin{proof}
For each $R_n$, let $\varphi_n:\Pic(R_n)\to\Pic(\O)$ denote the corresponding 
map in the exact sequence~$(\ref{exactsequence})$, and let 
$\phi_{n+1}\colon \Pic(R_{n+1})\to\Pic(R_n)$ send $[I]$ to $[IR_n]$, so that 
$\varphi_{n+1}=\varphi_n\circ\phi_{n+1}$.
These are all surjective group homomorphisms, and the underlying ideal maps preserve the norms of ideals prime to $p\idx{\O_K}{\O}$.
We assume $G_2=\Pic(R_2)$, which implies $G_n=\Pic(R_n)$ for $n\le 2$, and proceed by induction on $n$.

For each prime $q\in\mathcal{P}$ and every $n$, there are 
exactly $1-\inkron{d_K}{q}$ ideals in $R_n$ of norm~$q$, and $\phi_{n+1}$ 
maps $S_{n+1}$ onto $S_n$ and $G_{n+1}$ onto $G_n$.
By the inductive hypothesis, $G_n=\Pic(R_n)$, therefore $G_{n+1}$ intersects 
every coset of $\ker \phi_{n+1}\subset \ker \varphi_{n+1}$.
To prove $G_{n+1}=\Pic(R_{n+1})$, it suffices to show $\ker \varphi_{n+1}\subset G_{n+1}$.

The groups $\ker \varphi_n$ and $\ker \varphi_{n+1}$ are cyclic, by 
Lemma~\ref{cyclic}, since $p$ is odd and unramified.
Let $\alpha_n$ be a generator for $\ker \varphi_n$.  Since $\#\ker \varphi_n$ is divisible by $p$, $\alpha_n$ cannot be a $p$th power in $\ker\varphi_n$.
Expressing $\alpha_n$ in terms of $S_n$, we see that 
$\phi_{n+1}^{-1}(\alpha_n)$ must intersect $G_{n+1}$.
Let $\alpha_{n+1}$ lie in this intersection, and note that 
$\alpha_{n+1}\in\ker \varphi_{n+1}$.
The order of $\alpha_{n+1}$ must be a multiple of $|\alpha_n|=\#\ker\varphi_n$, and $\alpha_{n+1}$ cannot be a $p$th power in $\ker\varphi_{n+1}$.  It follows that $\alpha_{n+1}$ has order $\#\ker\varphi_{n+1}$, hence it generates $\ker \varphi_{n+1}$, proving $\ker\varphi_{n+1}\subset G_{n+1}$ as desired.
\end{proof}

To see Theorem~\ref{theorders} in action, let $\O$ be the order of discriminant $D=-7$, let $p=3$, and let $\mathcal{P}=\{2\}$.
The class group of the order $R_n$ of discriminant $3^{2n}D$ happens to be generated by an ideal of norm 2 when $n=2$,
and the theorem then implies that this holds for all $n$.  This allows us to construct arbitrarily large cyclic class groups, each generated by an ideal of norm 2.

We remark that Theorem \ref{theorders} may be extended to handle $p=2$ if the condition $G_2=\Pic(R_2)$ is replaced by $G_3=\Pic(R_3)$, and easily generalizes to treat families of orders lying in $\O$ that have $b$-smooth conductors, for any constant $b$.

\section{Explicit CM theory}\label{ExplicitCMTheory}

\subsection{The theory of complex multiplication (CM)}\label{CMTheory}
As in Section~3, let $\O$ be an order in a quadratic field $K$ of discriminant 
$d_K<-4$. We fix an algebraic closure of~$K$. It follows from class field 
theory that there is a unique field $K_\O$ with the property that the 
Artin map induces an isomorphism
$$
\Gal(K_\O/K) \isar \Pic(\O)
$$
between the Galois group of $K_\O/K$ and the ideal class group of $\O$.
The field $K_\O$ is called the {\it ring class field\/} for the order~$\O$.
If $\O$ is the maximal order of $K$, then $K_\O$ is the Hilbert class field 
of~$K$, the maximal totally unramified abelian extension of~$K$.
In general, primes dividing $\idx{\O_K}{\O}$ ramify in the ring 
class field. 

The first main theorem of complex 
multiplication~\cite[Thm.~11.1]{Cox:ComplexMultiplication} states that
$$
K_\O = K(j(E)),
$$
for any complex elliptic curve $E$ with endomorphism ring~$\O$.
Furthermore, the minimal polynomial $H_\O$ of $j(E)$ over $K$ actually has
coefficients in $\Z$, and its degree is~$h(\O) = |\Pic(\O)|$.
The polynomial $H_\O$ is known as the {\it Hilbert class polynomial\/}.
If $p$ is a prime that splits completely in the extension $K_\O/\Q$, 
then $H_\O$ splits into distinct linear factors in $\Fp[X]$.
Its roots are the $j$-invariants of the elliptic curves $E/\Fp$ with 
$\End(E)\cong\O$, a set we denote $\Ell_\O(\Fp)$.  Let $D=\disc(\O)$.
The primes that split completely in $K_\O$ are precisely the
primes $p \nmid D$ that are the norm
$$
p = N_{K/\Q}\left(\frac{t+v\sqrt{D}}{2}\right) = \frac{t^2 - v^2D}{4}
$$
of an element of~$\O$. The equation $4p = t^2 - v^2D$ is
often called the \emph{norm equation}.

For a positive integer~$N$, there is a unique extension $K_{N,\O}$ of the ring 
class field $K_\O$ such that the Artin map induces an isomorphism
$$
\Gal(K_{N,\O},K_\O) \isar (\O/N\O)^* / \{ \pm 1 \}.
$$
The field $K_{N,\O}$ is the {\it ray class field of conductor~$N$ for~$\O$\/}.
When $\O=\O_K$, this is simply the ray class field of conductor~$N$, and for
$N=1$ we recover the ring class field $K_\O = K_{1,\O}$.
The ring class field $K_{R}$ of the order $R = \Z + N\O$ is a subfield of the 
ray class field~$K_{N,\O}$. The Galois group of $K_R/K_\O$ is isomorphic to
$$
(\O/N\O)^* / (\Z/N\Z)^*,
$$
the kernel of the map $\varphi$ in~(\ref{exactsequence}).

The second main theorem of complex multiplication~\cite[Thm.~11.39]{Cox:ComplexMultiplication} states that
$$
K_{N,\O} = K_\O(x(E[N])),
$$
where $x(E[N])$ denotes the set of $x$-coordinates of the $N$-torsion points of an 
elliptic curve~$E$ with endomorphism ring~$\O$. 
The Galois invariance of the {\it Weil pairing\/} 
$E[N] \times E[N] \rightarrow \mu_N$ implies that the cyclotomic field 
$\Q(\zeta_N)$ is contained in the ray class field~$K_{N,\O}$ (a fact
that also follows directly from class field theory).
In particular, a prime $p$ that splits completely in $K_{N,\O}$ also splits 
completely in $\Q(\zeta_N)$, and is therefore congruent to $1$ modulo~$N$.

\subsection{Primes that split completely in the ray class field}\label{suitableprimes}
We are specifically interested in primes $p$ that split completely in the
ray class field $K_{l,\O}$, where $l$ is an odd prime. For such $p$
we can achieve the desired setting for Algorithm~\ref{alg1}, as depicted in
Figure 1.
\renewcommand\labelenumi{(\theenumi)}
\begin{theorem}\label{theprimes}
Let $l> 2$ be prime, and let $\O \not \subset \Z[i], \Z[\zeta_3]$ be an imaginary quadratic order that is maximal at~$l$.
Let $R = \Z+l\O$ be the order of index~$l$ in~$\O$.
Let $p$ be a prime that splits completely in the ray class field $K_{l,\O}$, but does not split completely in the ring class field for the order of index $l^2$ in~$\O$.
\begin{enumerate}
\item There are exactly $h(\O)$ different $\Fp$-isomorphism classes of elliptic curves $E/\Fp$ with endomorphism ring $\O$ that have $E[l] \subset E(\Fp)$.
\item There are exactly $h(R)$ different $\Fp$-isomorphism classes of elliptic curves $E/\Fp$ with endomorphism ring $R$ that have an $\Fp$-rational $l^2$-torsion point.
\end{enumerate}
\end{theorem}
\begin{proof}
The inclusions $K_\O \subseteq K_R \subseteq K_{l,\O}$ imply that both $H_\O$ 
and $H_R$ split into linear factors in~$\Fp$.
Each $j$-invariant in $\Ell_\O(\Fp)$, resp.~$\Ell_R(\Fp)$, corresponds to two 
distinct isomorphism classes over $\Fp$, since these curves are ordinary
and $\O \not = \Z[i], \Z[\zeta_3]$.
We will show that exactly one of these satisfies (1), resp.~(2). 

Since $p$ splits completely in the ray class field $K_{l,\O}$, we can 
factor $p = \pi_p \overline\pi_p \in \O$ with $\pi_p \equiv 1 \bmod l\O$.
Let $E/\Fp$ be an elliptic curve with endomorphism ring~$\O$ whose Frobenius 
endomorphism corresponds to $\pi_p$ under one of the two isomorphisms 
$\End(E) \isar \O$.
Since $\pi_p \equiv 1 \bmod l\O$, we have $E[l] \subset E(\Fp)$.
The Frobenius endomorphism of the non-isomorphic quadratic twist 
$\tilde{E}/\Fp$ corresponds to $-\pi_p$, and we then have 
$\#\tilde{E}(\Fp)=p+1-\tr(-\pi_p)\equiv 2\bmod l$.
For $l\not = 2$ this implies that~$\tilde{E}$ has trivial $l$-torsion 
over $\Fp$, proving (1).

To prove (2), let $E'/\Fp$ be a curve with endomorphism ring $R$ that is 
$l$-isogenous to~$E$. The Frobenius endomorphism of $E'$ also corresponds 
to $\pi_p$ under an isomorphism $\End(E') \isar R$.
The cardinality of $E'(\Fp)$ is thus equal to the cardinality of $E(\Fp)$ and 
therefore divisible by~$l^2$. However, since $p$ does {\it not\/} split 
completely in the ring class field of index $l^2$ in $\O$, we cannot have 
$\pi_p \equiv 1 \bmod lR$. It follows that $E'[l] \not \subset E'(\Fp)$ and 
$E'(\Fp)$ must contain a point of order $l^2$. As above, the quadratic twist 
of $E'$ must have trivial $l$-torsion over $\Fp$, proving (2).
\end{proof}

Provided the order $\O$ in Theorem~\ref{theprimes} also satisfies 
$h(\O)\ge l+2$, we can achieve the desired setting for Algorithm~\ref{alg1}.
We say such an order is \emph{suitable} for $l$.

To determine the coefficients $\Phi_l$ via the Chinese Remainder Theorem, we 
need to compute $\Phi_l\bmod p$ for many primes $p$ satisfying 
Theorem~\ref{theprimes}.
We necessarily have $p>l$, since $p\equiv 1\bmod l$, and the height 
bound~\cite{CohenPaula:ModularPolynomials} on the coefficients of $\Phi_l$ 
implies that $6l+O(1)$ primes suffice.
We now show these primes exist and bound their size, assuming the GRH.
For this purpose we define a \emph{suitable family of orders}.

\begin{definition}\label{suitableorders}
{Let $S$ be the set of odd primes and let $T$ be the set of all imaginary quadratic orders.
A suitable family of orders is a function $\mathcal{F}:S\to T$ such that:
\begin{enumerate}
\item
for all $l\in S$ the order $\mathcal{F}(l)$ is suitable for $l$.
\item
there exist effective constants $c_1,c_2\in\R_{>0}$ such that for all $l\in S$ the bounds $l+2\le h(\mathcal{F}(l))\le c_1l$ and $l^2\le|\disc(\mathcal{F}(l))|\le c_2l^2$ hold.
\end{enumerate}
}
\end{definition}

\begin{example}\label{exam}
{Let $\mathcal{F}(3)=\Z[\sqrt{-47}]$, and for $l > 3$ 
let $\mathcal{F}(l)$ be the order $\O$ of discriminant $-7\cdot3^{2n}$, 
where $n$ is the least integer for which $h(\O)=2\cdot3^{n-1}\ge l+2$.
Letting $c_1=4$ and $c_2=205$, we see that $\mathcal{F}$ is a suitable family 
of orders.}
\end{example}

\begin{theorem}\label{smallprimes}
Let $\mathcal{F}$ be a suitable family of orders and let $c_0\in \R_{>0}$ be an arbitrary constant.
Then for each prime $l>2$ the set of primes $p$ for which $l$, $\O=\mathcal{F}(l)$, and~$p$ satisfy the conditions of Theorem~\ref{theprimes} has positive density.

Assuming the GRH, there is an effective constant $c\in\R_{>0}$ such that at 
least $c_0l^3(\log l)^3$ of these primes are bounded by $B=cl^6(\log l)^4$, 
for all primes $l>7$.
\end{theorem}
\begin{proof}
For a prime $l>2$, let $\O=\mathcal{F}(l)$ have fraction field $K$, and let $u=[\O_K:\O]$.
The ray class field $K_{l,\O}$ and the ring class field $K_S$ for the order $S=\Z+l^2\O$ are both invariant under the action of
complex conjugation, hence both are Galois extensions of $\Q$.
One finds that
\begin{center}
$\#\Gal(K_{l,\O}/\Q)=2\bigl(l-1\bigr)\bigl(l-\inkron{d_K}{l}\bigr)h(\O)\medspace < \medspace 2l\bigl(l-\inkron{d_K}{l}\bigr)h(\O) = \#\Gal(K_S/\Q),$
\end{center}
and the Chebotar\"ev density theorem~\cite[Thm.~13.4]{Neukirch:AlgebraicNumberTheory} yields the unconditional claim.

To prove the conditional claim, we apply an effective Chebotar\"ev bound to the extension $K_{l,\O}/\Q$, assuming the GRH for the Dedekind zeta function of $K_{l,\O}$.

The extension $K_{l,\O}/K$ is abelian of conductor dividing $lu$, with 
degree $nh(\O)$, where $n\leq 2 \#(\O/l\O)^*\le 2l^2$.
The $\O_K$-ideal $\disc(K_{l,\O}/\O)$ is a divisor of $(lu)^{nh(\O)}$, by 
Hasse's {\it F\"uhrerdiskriminantenproduktformel\/}~\cite[Thm.~VII.11.9]{Neukirch:AlgebraicNumberTheory}.
We then have
\begin{align*}
|\disc(K_{l,\O}/\Q)| &= |N_{K/\Q}(\disc(K_{l,\O}/K))\cdot 
                         \disc(K/\Q)^{[K_{l,\O}:K]} |\\
   &\le (lf)^{2 nh(\O)} |\disc(K/\Q)|^{nh(\O)} \le (c_2l^4)^{nh(\O)},
\end{align*}
where $\disc(\O)\le c_2l^2$.  Using the bound $h(\O)\le c_1l$, Theorem~1.1 of~\cite{Lagarias:Chebotarev} then yields
\begin{equation}\label{LCbound}
\left|\pi(x,K_{l,\O}/\Q)-\frac{\Li(x)}{2nh(\O)}\right|\le c_3\left(x^{1/2}\log(lx)+l^3\log l\right),
\end{equation}
where $\pi(x,K_{l,\O}/\Q)$ counts the primes up to $x\in\R_{>0}$ that split 
completely in $K_{l,\O}$, and $c_3\in \R_{>0}$ is an effectively computable 
constant, independent of $l$.

If we now suppose $x=cl^6(\log l)^4$, and apply $\Li(x)\sim x\log x$ and 
$nh(\O)\le c_1l^3$, we may choose $c\in \R_{>0}$ so that $\Li(x)/(2nh(\O))$ 
is greater than the RHS of~(\ref{LCbound}) by an arbitrarily large constant 
factor. In particular, for any $c_4\in R_{>0}$ there is an effectively 
computable choice of $c$ that ensures $\pi(x,K_{l,\O}/\Q) \ge c_4l^3(\log l)^3$, independent of $l$.
Moreover, for the least such $c$ we have $c/c_4\to 1$ as $c_4\to\infty$.

We now show that most of these primes do not split completely in $K_S$.
Any prime $p$ that splits completely in $K_{l,\O}$ must split completely in the ring class field for $R=\Z+l\O$.
Putting $D=\disc(\O)$, we then have
\begin{equation}\label{normeq1}
4p=t^2-v^2l^2D,
\end{equation}
with $t,v\in\Z_{>0}$ and $t\equiv 2\bmod l$.
If $v\not\equiv 0\bmod l$, then $p$ cannot split completely in $K_S$.
For $p\le cl^6(\log l)^4$, we have $v\le 2c^{1/2}l(\log l)^2$ and 
$t\le 2c^{1/2}l^3(\log l)^2$, since $D\ge l^2$, hence there are at most 
$2c^{1/2}(\log l)^2$ positive $v\equiv 0\bmod l$, and at most 
$2c^{1/2}l^2(\log l)^2+1$ positive $t\equiv 2\bmod l$, that 
satisfy~(\ref{normeq1}).

It follows that no more than $4cl^2(\log l)^4+2c^{1/2}(\log l)^2$ primes 
$p\le cl^6(\log l)^4$ split completely in $K_S$. For a sufficiently large 
choice of $c_4$, we can choose $c$ so that 
$\pi(x,K_{l,\O}/\Q) = \pi (cl^6(\log l)^4,K_{l,\O}/\Q) \ge c_4l^3(\log l)^3$ and also
$$
c_4l^3(\log l)^3-4cl^2(\log l)^4-2c^{1/2}(\log l)^2>c_0l^3(\log l)^3,
$$
provided $l > 7$, since $l/(\log l)$ is bounded above 4 for primes $l > 7$.
\end{proof}

Theorem~\ref{smallprimes} guarantees we can obtain a sufficient number of primes $p$ for use with Algorithm~\ref{alg1}.
In fact, as is typical for such bounds, it provides far more than we need.  The task of finding these primes is addressed in Section~\ref{selectingprimes}.

\subsection{Computing the CM action}\label{CMaction1}\par
The Galois action of $\Gal(K_\O/K)\cong\Pic(\O)$ on the set
$\Ell_\O(\Fp)$ may be explicitly computed using isogenies,
as described in \cite{Belding:HilbertClassPolynomial}.
Let the prime $p$ split completely in the ring class field $K_\O$, 
and let $E/\Fp$ be an elliptic curve with $\End(E)\cong\O$.
Fixing an isomorphism $\End(E)\isar\O$, for each invertible $\O$-ideal $\mf{a}$ we define
$$
E[\mathfrak{a}] = \{ P \in E(\overline \F_p) \mid \forall \tau \in \mathfrak{a} : \tau(P) = 0 \},
$$
the `$\mf{a}$-torsion' subgroup of $E$.
The subgroup $E[\mf{a}]$ is the kernel of a separable isogeny $E\rightarrow E/E[\mf{a}]$ of degree $\idx{\O}{\mf{a}}$, with $\End(E/E[\mf{a}])\cong \O$.  This yields a group action
$$
j(E)^\mathfrak{a} = j(E/E[\mathfrak{a}]),
$$
in which the ideal group of $\O$ acts on the set $\Ell_\O(\Fp)$.
This action factors through the class group, and the $\Pic(\O)$-action is transitive and free.  Equivalently,
$\Ell_\O(\Fp)$ is a \emph{torsor} for $\Pic(\O)$; for each pair $(j_1,j_2)$ of elements in $\Ell_\O$ there is a unique
element of $\Pic(\O)$ whose action sends $j_1$ to $j_2$.

Now let $\mfl_0$ be an invertible $\O$-ideal of prime norm $l_0\not= p$.
The curves $E$ and $E/E[\mfl_0]$ are $l_0$-isogenous, hence
$$
\Phi_{l_0}(j_0,j_0^{{\mfl}_0})= 0,
$$
where $j_0=j(E)$.
To compute the action of $\mfl_0$, we need to find the corresponding root of $\Phi_{l_0}(X,j_0)\in\Fp[X]$.
We assume that $\Phi_{l_0}(X,Y)$ is known, either via one of the algorithms 
from the introduction, or by a previous application of Algorithm~\ref{alg2}.
The polynomial $\Phi_{l_0}(X,j_0)\in\Fp[X]$ has either 1 or 2 roots that lie 
in $\Ell_\O(\Fp)$, depending on whether $l_0$ ramifies or splits (it is not inert).
These roots correspond to the actions of $\mfl_0$ and its inverse $\mfl_0^\invexp$, which coincide when $l_0$ ramifies. 

Our fixed isomorphism $\End(E)\isar\O$ maps the Frobenius endomorphism of $E$ to an element $\pi_p\in\O\subset\O_K$ with norm $p$.
We then have the norm equation
\begin{equation}\label{norm-equation}
4p=t^2-v^2d_K,
\end{equation}
where $t=\tr(\pi_p)$, and $v$ is the index of $\Z[\pi_p]$ in $\O_K$.
When $l_0$ does not divide $v$, the order $\Z[\pi_p]$ is 
maximal at $l_0$ and the only roots of $\Phi_{l_0}(X,j_0)$ over $\Fp$ are those 
in $\Ell_\O(\Fp)$.
Otherwise $\Phi_{l_0}(X,j_0)$ has $l+1$ roots in $\Fp$, and those in $\Ell_\O(\Fp)$ lie on the surface of the $l_0$-volcano containing~$j$, as described in \cite{Fouquet:IsogenyVolcanoes}.
The roots on the surface can be readily distinguished, as in~\cite[\S4]{Sutherland:HilbertClassPolynomials}, for example, but typically we choose $p$ with $l_0\nmid v$ so that every root of $\Phi_{l_0}(X,j_0)$ in $\Fp$ is on the surface.

When $l_0$ splits and does not divide $v$, the actions of $\mfl_0$ 
and $\mfl_0^\invexp$ may be distinguished as described in~\cite[\S5]{Broker:pAdicClassPolynomial} and~\cite[\S3]{Galbraith:GHSattack}.
The kernels of the two $l_0$-isogenies are subgroups of $E[l_0]$.
A standard component of the SEA algorithm computes a polynomial $F_{l_0}(X)$, whose roots are the abscissa of the points in one of these kernels~\cite{Elkies:AtkinBirthday,Schoof:ECPointCounting2}.
In our setting $l_0$ splits in $\Z[\pi_p]$, and provided $l_0\nmid v$, the action of $\pi_p$ on $E[l_0]$ has two distinct eigenvalues corresponding to the two kernels.
Expressing the ideal $\mfl_0$ in the form $(l_0,c+d\pi_p)$ yields the eigenvalue $\lambda=-c/d \bmod l_0$.
We may then use $F_{l_0}(X)$ to test whether $\pi_p$'s action is equivalent to multiplication by $\lambda$ in the corresponding kernel.  See~\cite{Broker:pAdicClassPolynomial} for an example and further details.

As a practical optimization (see Section~\ref{selectorder}), we avoid the need to ever make this distinction.
The asymptotic complexity of computing the action of $\mfl_0$ is the same in any case.

\begin{lemma}\label{CMactioncost}
Let $l_0$ and $p$ be distinct odd primes, and let $\mathcal{O}\ne\Z[i],\Z[\zeta_3]$ be an imaginary quadratic order.
Let $j_0=j(E)\in\Ell_\O(\Fp)$, fix an isomorphism $\End(E)\isar\O$, and let $\pi_p\in\O$ denote the image of the Frobenius endomorphism.
Let $\mfl_0$ be an invertible $\O$-ideal of norm $l_0$, and assume $\Z[\pi_p]$ is maximal at $l_0$.

Given $\Phi_{l_0}\in\Fp[X,Y]$, the $j$-invariant $j_0^{\mfl_0}$ may be computed using an expected
$O(l_0^2+\text{\emph{$\M(l_0)$}}\log p)$ operations in $\Fp$.
\end{lemma}
Here $\M(n)$ denotes the complexity of multiplying two polynomials of degree less than $n$, as in~\cite[Def.~8.26]{Gathen:ComputerAlgebra}.
Na\"ively, $\M(n)=O(n^2)$, Karatsuba's algorithm yields $\M(n)=O(n^{\log_2 3})$, and methods based on the fast Fourier transform (FFT) achieve $\M(n)=O(n\log n\log\log n)$.
\begin{proof}
We first compute $\gcd(X^p-X,\Phi_{l_0}(X,j_0))$, the product of the
distinct linear factors of $\Phi_{l_0}(X,j_0)$ over $\Fp$.  Instantiating
$f(X)=\Phi_{l_0}(X,j_0)$ uses $O(l_0^2)$ operations in $\Fp$, exponentiating
$X^p\bmod f$ uses $O(\M(l_0)\log p)$ operations in $\Fp$, and the fast
Euclidean algorithm~\cite[\S11.1]{Gathen:ComputerAlgebra} obtains
$\gcd(X^p-X,f)$ using $O(\M(l_0)\log l_0)=O(l_0^2)$ operations in $\Fp$.
This gcd has degree at most 2, since $\Z[\pi_p]$ is maximal to $l_0$,
and we may find its roots using an expected $O(\log p)$ $\Fp$-operations
\cite[Cor.~14.16]{Gathen:ComputerAlgebra}.

The desired root $j_0^{\mfl_0}$ is then distinguished as outlined above.
We first compute the eigenvalue $\lambda-c/d\bmod l_0$, where $\mfl_0=(l_0,c+d\pi_p)$, using $O(l_0^2)$ bit operations.
Applying~\cite[Thm.~2.1]{Bostan:FastIsogenies}, the kernel polynomial $F_{l_0}(X)$ can be computed using $O(\M(l_0))$ operations in $\Fp$.
To compare $(X^p,Y^p)$ to the scalar multiple $\lambda\cdot(X,Y)$, we compute $X^p$, $Y^p$, and the required division polynomials $\psi_n(X,Y)$, modulo $F_{l_0}(X)$ and the curve equation for $E$, as in the SEA algorithm~\cite[Ch.~VII]{Blake:EllipticCurves}.
This uses $O((\log l_0 + \log p)\M(l_0))=O(l_0^2+\M(l_0)\log p)$ operations in $\Fp$.
\end{proof}

\section{Mapping the CM torsor}
The previous section made explicit the Galois action corresponding to an
element of $\Pic(\O)\cong\Gal(K_\O/K)$ represented by an ideal $\mathfrak{l}_0$ of prime norm $l_0$.
We now use this to enumerate the set $\Ell_\O(\Fp)$, and at the same time compute a map that
explicitly identifies the action of each element of $\Pic(\O)$.
To do this efficiently it is critical to work with generators whose norms are small, since
the cost of computing the action of $\mathfrak{l}_0$ increases quadratically with its norm.

\subsection{Polycyclic presentations}\label{polycyclic}
As a finite abelian group, each element of $\Pic(\O)$ can be uniquely represented using a basis.
However, as noted in~\cite[\S5.3]{Sutherland:HilbertClassPolynomials}, the norms arising in
a basis may need to be much larger than those in a set of generators.
Thus we are led to consider polycyclic presentations.

Let $\vec{\alpha}=(\alpha_1,\ldots,\alpha_k)$ be a sequence of generators for $G$, and let $G_i=\langle\alpha_1,\ldots,\alpha_i\rangle$ denote the subgroup generated by $\alpha_1,\ldots,\alpha_i$.
The composition series
$$
1=G_0\le G_1 \le \cdots\le G_{k-1} \le G_{k} = G,
$$
is then polycyclic, meaning that each quotient $G_{i+1}/G_i$ is a cyclic group.  The sequence $r(\vec{\alpha})=(r_1,\ldots,r_k)$ of \emph{relative orders} for $\vec{\alpha}$ is defined by
$$
r_i=|G_i:G_{i-1}|.
$$
Each $r_i$ necessarily divides $|\alpha_i|$, and for $i>1$ we typically have $r_i < |\alpha_i|$.
The sequences $\vec{\alpha}$ and $r(\vec{\alpha})$ allow us to uniquely represent each $\beta\in G$ in the form
\begin{equation}\label{pcprep}
\beta=\vec{\alpha}^\vec{x}=\alpha_1^{x_1}\cdots\alpha_k^{x_k},
\end{equation}
where $\vec{x}=(x_1,\ldots,x_k)$ with $0\le x_i <r_i$.
The vector $s(\vec{\alpha},i)=\vec{x}$ for which $\alpha_i^{r_i}=\vec{\alpha}^\vec{x}$ has $x_j=0$ for $j\ge i$, and is called a \emph{power relation}, see \cite[\S8.1]{Holt:CGTHandbook}.
A generic algorithm to compute $r(\vec{\alpha})$ and the $s(\vec{\alpha},i)$ can be found in~\cite[Alg.~2.1]{Sutherland:HilbertClassPolynomials}.

The vector $\vec{x}$ in (\ref{pcprep}) is the discrete logarithm or \emph{exponent vector} of $\beta$.  We let
$$
X(\vec{\alpha})=\{\vec{x}\in\Z^k:0\le x_i<r_k\},
$$
and note that the map $\vec{x}\mapsto\vec{\alpha}^\vec{x}$ defines a bijection from $X(\vec{\alpha})$ to $G$.
\medskip

We now consider the case $G=\Pic(\O)$, where $\O$ is an order in a quadratic field $K$ of discriminant $d_K<-4$.
Let $\mathcal{P}=(p_1,p_2,p_3,\ldots)$ be an increasing sequence of primes with the property that $\Pic(\O)$ is generated by the classes of invertible ideals with norms in $\mathcal{P}$.  
By Dirichlet's density theorem~\cite[Thm.~9.12]{Cox:ComplexMultiplication}, any sequence containing all but a finite set of primes works,
and from~\cite[Cor.~7.17]{Cox:ComplexMultiplication} we know that some finite prefix of $\mathcal{P}$ actually suffices.
For $\O=\O_K$ we may take $\mathcal{P}$ to be the sequence of primes less than $|d_K/3|^{1/2}$, by~\cite[Prop.~9.5.2]{Buchmann:BinaryQuadraticForms}.

There is a unique lexicographically minimal subsequence $(l_1,\ldots,l_k)$ of $\mathcal{P}$ that corresponds to 
a polycyclic sequence $\vec{\alpha}=(\alpha_1,\ldots,\alpha_k)$ for $\Pic(\O)$ in which $\alpha_i$ is represented by an ideal of norm $l_i$
and $r(\vec{\alpha})$ has $r_i>1$.
When $l_i$ splits there are two possibilities for $\alpha_i$.
To fix a choice, let $\alpha_i$ be the ideal class represented by the unique binary quadratic form $ax^2+bxy+cy^2$ of discriminant $D=\disc(\O)$ with $a=l_i$ and $b$ nonnegative~\cite[\S3.4]{Buchmann:BinaryQuadraticForms}, corresponding to the ideal $\mfl_i=(l_i,(-b+\sqrt{D})/2)$.

We call $\vec{\alpha}$ the \emph{polycyclic presentation} of $\Pic(\O)$ determined by $\mathcal{P}$.
We use $l(\vec{\alpha})$ to denote the sequence of norms $(l_1,\ldots,l_k)$, but note that this also depends on $\mathcal{P}$; each $l_i$ is the least prime in $\mathcal{P}$ that is the norm of an ideal in $\alpha_i$.

We may compute $\vec{\alpha}$ by applying~\cite[Alg.~2.1]{Sutherland:HilbertClassPolynomials} to an implicit sequence of generators $\vec{\gamma}=(\gamma_1,\gamma_2,\gamma_3,\ldots)$ corresponding to the subsequence of $\mathcal{P}$ for which there exists an invertible $\O$-ideal of norm $p_i$.
The algorithm computes $r_i$ for each $\gamma_i$ in turn, and if we find that $r_i>1$, we append $\gamma_i$ to an initially empty vector~$\vec{\alpha}$.  We terminate when $\prod r_i=h(\O)$, a value which we assume has been precomputed.

The computation of $\vec{\alpha}$ uses $|G|=h(\O)$ group operations in $G=\Pic(\O)$, and creates a table $T:X(\vec{\alpha})\to G$ that stores $|G|=h(\O)$ group elements~\cite[Prop.~6]{Sutherland:HilbertClassPolynomials}.
Using binary quadratic forms to represent $\Pic(\O)$, the group operation has bit-complexity $O(\log^2|\disc(O)|)$, as shown in~\cite{Biehl:FormReductionComplexity}, and each element may be stored in $O(\log|\disc(\O)|)$ space.
Evaluating $T(\vec{x})$, or $T^{-1}(\beta)$, has bit-complexity $O(\log|G|)$.

\subsection{Suitable presentations}
When $\mathcal{P}$ is the sequence of all primes, the norms $l(\vec{\alpha})$ for the polycyclic presentation $\vec{\alpha}$ of $\Pic(\O)$ determined by $\mathcal{P}$ are as small as possible.
However, when working in the finite field $\Fp$ we may wish to ensure that each norm $l_i$ does not divide $v=[\O_K:\Z[\pi_p]]$, as noted in Section~\ref{CMaction1}.
This is achieved by excluding from $\mathcal{P}$ primes that divide $v$, and we call the corresponding $\vec{\alpha}$ the presentation of $\Pic(\O)$ \emph{suitable} for $p$.  This may cause us to use norms that are slightly larger than optimal.  We now show that, provided we work
with a family of orders that satisfies certain (easily met) constraints, the norms in every suitable presentation are quite small, assuming the GRH.

\begin{theorem}\label{suitablepresentation}
Let $c_3\in\R_{>0}$ be a fixed constant, and let $\mathcal{F}$ be a suitable family of orders with the following additional property: if $\O$ is an order in $\mathcal{F}$ whose fraction field $K$ has discriminant $d_K$, and $s_\O$ denotes the square-free part of $[\O_K:\O]$, then $s_\O$ is coprime to $2d_k$ and both $s_\O$ and $|d_K|$ are bounded by $c_3$.

Then under the GRH, for every $\O$ in $\mathcal{F}$ and every prime $p$ that splits completely in $K_\O$, the presentation $\vec{\alpha}$ of $\Pic(\O)$ suitable for $p$ has norms $l(\vec{\alpha})$ for which
$$\max l(\vec{\alpha})\le c\omega(v)\log (\omega(v)+1),$$
where $v$ is defined by $4p=t^2-v^2\disc(\O)$, the function $\omega(v)$ counts the distinct prime factors of $v$, and $c$ is an effective constant that depends only on $c_3$.
\end{theorem}
\begin{proof}
Let $\O$, $p$, $v$, and $\vec{\alpha}$ be as above.
Let $R$ be the order of index $s_\O^2$ in $\O_K$, and let $\vec{\gamma}$ be the presentation of $R$ determined by the increasing sequence of primes that do not divide $v$.
It follows from Theorem~\ref{theorders} that $\max l(\vec{\alpha}) \le \max l(\vec{\gamma})$.

Let $K_R$ be the ring class field for $R$, and let $\pi_C(x,K_R/\Q)$ count the primes bounded by $x\in\R_{>0}$ whose Frobenius symbol (under the Artin map) lies in the conjugacy class $C$ of $\Gal(K_R/\Q)$.
We may bound $\idx{K_R}{\Q}$, $\#\Gal(K_R/\Q)$, and $\disc(K_R)$ by constants that depend only on $c_3$, independent of $\mathcal{O}$.
Under the GRH, the Chebotar\"ev bound of~\cite[Thm.~1.1]{Lagarias:Chebotarev} then yields
$$
\pi_C(x,K_R/\Q) \ge c_4x/\log x,
$$
for some effective constant $c_4\in\R_{>0}$ and all $x > 2$, where $c_4$ depends only on~$c_3$.
For an effective constant $c$ depending on $c_3$, setting $x=c\omega(v)\log(\omega(v)+1)$ yields $\pi_C(x,K_R/\Q) > \omega(v)$.  In this case the Frobenius symbol of at least one prime not dividing $v$ lies in $C$, and this applies to every $C$.  It follows that every class in $\Pic(R)$ contains an element whose norm is a prime bounded by $x$ that does not divide $v$.
We then have $\max l(\vec{\alpha}) \le \max l(\vec{\gamma}) \le x=c\omega(v)\log(\omega(v)+1)$, as desired.
\end{proof}

The family of orders in Example~\ref{exam} satisfies the requirements of Theorem~\ref{suitablepresentation}.
We note that provided $\log p = O(\log l)$, we have $\omega(v)=O(\log l/\log\log l)$, and the theorem then yields an $O(\log l)$ bound on the norms $l(\vec{\alpha})$.
This is sharper than the more general $O(\log^2 l)$ bound implied by~\cite{Bach:ERHbounds}.
In fact, by~\cite[Thm.~431]{Hardy:NumberTheory}, one expects $\omega(v)=O(\log\log p)$, which yields a bound of $O(\log\log l \log\log\log l)$.

\subsection{Realizing the CM torsor}\label{CMaction2}
We now consider how to explicitly map $\Pic(\O)$ to the torsor $\Ell_\O(\Fp)$, so that we may then compute the action of any element or subgroup of $\Pic(\O)$ on any element of $\Ell_\O(\Fp)$, \emph{without needing to compute any further isogenies}.
We use the presentation $\vec{\alpha}$ of $\Pic(\O)$ suitable for~$p$, and the table $T:X(\vec{\alpha})\to\Pic(\O)$ described in Section~\ref{polycyclic}.
As above, we have $\vec{\alpha}=([\mfl_1],\ldots,[\mfl_k])$, with norms $l(\vec{\alpha})=(l_1,\ldots,l_k)$ and relative orders $r(\vec{\alpha})=(r_1,\ldots,r_k)$. We assume the modular polynomials $\Phi_{l_1},\ldots,\Phi_{l_k}$ are known, since the $l_i$ are small.

To enumerate $\Ell_\O(\Fp)$ we use~\cite[Alg.~1.3]{Sutherland:HilbertClassPolynomials}, but we augment this algorithm to also compute an explicit bijection $\phi\text{\hspace{2pt}:\hspace{1pt}}\Pic(\O)\to\Ell_\O(\Fp)$ in which $[\mathfrak{a}]\in\Pic(\O)$ corresponds to $j_0^{\mathfrak{a}}$.
Given $j_0\in\Ell_\O(\Fp)$, we compute a path of $l_k$-isogenies
\begin{equation}\label{isogenypath}
j_0\mapright{\mfl_k}j_1\mapright{\mfl_k}j_2\mapright{\mfl_k}\cdots\mapright{\mfl_k}j_{r_k-1},
\end{equation}
where the $j_i$ are distinct elements of $\Ell_\O(\Fp)$.
As explained in Section~\ref{CMaction1}, each step in this path is computed by finding a root $j_i$ of $\Phi(X,j_{i-1})$ that lies in $\Ell_\O(\Fp)$.
When $l_k$ splits in $K$ we have two choices for $j_1$, and the correct choice may be determined using a kernel polynomial as outlined in Section~\ref{CMaction1}.  For $i>1$ we use the polynomial $\Phi(X,j_{i-1})/(X-j_{i-2})$, which has exactly one root $j_i\in\Ell_\O(\Fp)$.

When $k>1$, the enumeration of $\Ell_\O(\Fp)$ proceeds recursively:
for each $j_i$ in (\ref{isogenypath}) we compute a path of $l_{k-1}$ isogenies containing $r_{k-1}$ distinct $j$-invariants.
Eventually, every element of $\Ell_\O(\Fp)$ is enumerated exactly once~\cite[Prop.~5]{Sutherland:HilbertClassPolynomials}.
For each $j_n\in\Ell_\O(\Fp)$ we also compute a vector $\vec{x}\in X(\vec{\alpha})$ that describes the path used to reach $j_n$ from $j_0$, where $x_i$ indicates the number of steps taken on an $l_i$-isogeny path.
By correctly choosing the direction of each path, we ensure that each $j_n$ is the image of $j_0$ under the $\Pic(\O)$-action of $\vec{\alpha}^{\vec{x}}=\alpha_1^{x_1}\cdots\alpha_k^{x_k}$.
This yields the desired bijection $\phi$; since $\vec{\alpha}^{\vec{x}}$ uniquely represents some $\beta\in\Pic(\O)$, we may set $\phi(\vec{\alpha}^{\vec{x}})=j_n$, a process facilitated by the map $T:X(\vec{\alpha})\to\Pic(\O)$.

The bijection $\phi$ allows us translate any computation in the group $\Pic(\O)$ to the torsor $\Ell_\O(\Fp)$.  In particular, by enumerating the cyclic subgroup $H\subseteq\Pic(\O)$ generated by $[\mfl]$, where $\mfl$ is an ideal of norm $l$, we obtain the $l$-isogeny cycle containing~$j_0$, corresponding to the surface of one of the $l$-volcanoes in Figure~1.  Doing the same for each coset of $H$ partitions $\Ell_\O(\Fp)$ into $l$-isogeny cycles.

This may also be applied to the order $R=\Z+l\O$.  After obtaining a bijection from $\Pic(R)$ to $\Ell_R(\Fp)$, we enumerate the kernel of the map $\varphi\text{\hspace{2pt}:\hspace{1pt}}\Pic(R)\to\Pic(O)$ from the exact sequence of (\ref{exactsequence}).  Here we use one of the generators of norm $l^2$ guaranteed by Lemma~\ref{generators}.  Enumerating the cosets of $\ker \varphi$ then partitions $\Ell_R(\Fp)$ into $l^2$-isogeny cycles of siblings with a common $l$-isogenous parent in $\Ell_\O(\Fp)$.

\section{The algorithm}
We now present our algorithm to compute the modular polynomial $\Phi_l$ using the Chinese Remainder Theorem (CRT).
Algorithm~\ref{alg2} follows the standard pattern of a CRT-based algorithm; the details lie in Algorithm~\ref{alg21}, which selects a set of primes $S$, and in Algorithm~\ref{alg1}, which computes $\Phi_l$ modulo each prime $p\in S$.

The computation of $\Phi_l\in\Z[X,Y]$ may be viewed as a special case of computing $\Phi_l\in(\Z/m\Z)[X,Y]$, where $m$ is the product of the primes in $S$.  The choice of $S$ ensures that this $m$ is large enough to uniquely determine $\Phi_l\in\Z[X,Y]$.

\renewcommand\labelenumi{\theenumi.}
\renewcommand\labelenumii{\theenumii.}
\begin{algorithm}\label{alg2}
{Let $l$ be an odd prime, let $m$ be a positive integer, and let $\O=\mathcal{F}(l)$ lie in a suitable family of orders $\mathcal{F}$.
Compute $\Phi_l\in(\Z/m\Z)[X,Y]$ as follows:}
\begin{enumerate}
\item
Compute the Hilbert class polynomial $H_\O\in\Z[X]$.
\item
Select a set of primes $S$ with Algorithm~\ref{alg21}, using $l$ and $\O$.
\item
Perform CRT precomputation using $S$.
\item
For each prime $p\in S$:
\begin{enumerate}
\item
Compute $\Phi_l\bmod p$ with Algorithm~\ref{alg1}, using $\O$ and $H_\O$.
\item
Update CRT data using $\Phi_l\bmod p$.
\end{enumerate}
\item
Perform CRT postcomputation.
\item
Output $\Phi_l\in(\Z/m\Z)[X,Y]$.
\end{enumerate}
\end{algorithm}
The suitable family of orders $\mathcal{F}$ is as defined in Section~\ref{suitableprimes}, see Definition~\ref{suitableorders}.
The polynomial $H_\O$ computed in Step~1 may be obtained using any of several algorithms whose running time is quasi-linear in $\disc(\O)$, including~\cite{Broker:pAdicClassPolynomial,Enge:FloatingPoint,Sutherland:HilbertClassPolynomials}.

\subsection{Selecting primes}\label{selectingprimes}
The primes $p$ in the set $S$ selected by Algorithm~\ref{alg2} must satisfy the conditions of Theorem~\ref{theprimes} in order to use them in Algorithm~\ref{alg1}.
We require $p$ to split completely in the ray class field $K_{l,\mathcal{\O}}$, but to not split completely in the ring class field for the order $\Z+l^2\O$.  Equivalently, we need $p\nmid D$ to satisfy
\begin{equation}\label{normeq2}
4p=t^2-v^2l^2D,
\end{equation}
with $t\equiv 2\bmod l$ and $l\nmid v$, where $D=\disc(\O)$.  To apply the CRT, we also require
\begin{equation}\label{primeset}
\prod_{p\in S}\log p\ge 4|c|,
\end{equation}
for every coefficient $c$ of $\Phi_l\in\Z[X,Y]$.  From~\cite{BrokerSutherland:PhiHeightBound}, we use the explicit bound
\begin{equation}\label{heightbound}
B_l = 6l\log l + 18l
\end{equation}
on the logarithmic height of $\Phi_l$ to achieve this.  We then have $\#S=O(l)$, by (\ref{normeq2}).

Heuristically, it is easy to find primes that satisfy (\ref{normeq2}).  If $D\equiv 1\bmod 8$, fix $v=2$, otherwise fix $v=1$.
Then, for increasing $t\equiv 2\bmod l$ with the correct parity, test whether $p=(t^2-v^2l^2D)/4$ is prime.
We expect to need $O(l\log l)$ primality tests, and each can be accomplished in time polynomial in $\log l$,
although typically $p$ is small enough to make an attempted factorization more efficient.
We could obtain slightly smaller $p$'s by letting $v$ vary, but it is more convenient to fix $v$
so that we can use the same presentation of $\Pic(\O)$ and $\Pic(R)$ for every $p$.
This approach is easy to implement and very fast in practice.
\medskip

However, in order to prove Theorem~\ref{main-thm} we must take a more cautious approach.
Even assuming the GRH, we cannot guarantee we will find \emph{any} primes with a fixed value of $v$.
On the other hand, Theorem~\ref{smallprimes} implies that if we construct random integers $p\le x$ satisfying (\ref{normeq2}), for sufficiently large $x$ we have $p$ prime with probability $\Omega(1/\log x)$, and under the GRH, $x=O(l^6(\log l)^4)$ is large enough.
Additionally, we would like to avoid $v$'s with many prime factors, so that we may more profitably apply Theorem~\ref{suitablepresentation}.
The restriction $\omega(v)\le 2\log(\log v+3)$ eliminates an asymptotically negligible proportion of the integers $v\in[1,x]$ (see Lemma~\ref{HWbound}).
\smallbreak

We now present Algorithm~\ref{alg21}, emphasizing that its purpose is to facilitate the proof of Theorem~\ref{main-thm}.
In practice we use the heuristic procedure described above.

\renewcommand\labelenumi{\theenumi.}
\renewcommand\labelenumii{\theenumii.}
\begin{algorithm}\label{alg21}
{Let $l$ be an odd prime, let $D<-4$ be a discriminant, and let $B_l$ be as in (\ref{heightbound}).
Construct the set $S$ as follows:}
\begin{enumerate}
\algitem
Set $n\leftarrow (B_l+2\log 2)/\log(l^2|D|/4)$ and then $x\leftarrow 4l^2|D|n\log n$.\\
Set $b\leftarrow 0$ and $S\leftarrow \emptyset$.
\algitem
Set $T\leftarrow 2x^{1/2}$ and $V\leftarrow 2x^{1/2}l^{-1}|D|^{-1/2}$.
\algitem
Repeat $\lceil 2N\log x\rceil $ times:
\begin{enumerate}
\item
Construct an integer $p=(t^2-v^2l^2D)/4$ using uniformly random integers $v\in[1,V]$ and $t\in[1,T]$, subject to $l\nmid v$, $t\equiv 2\bmod l$, and $t\equiv vD\bmod 2$.
\item
If $\omega(v)> 2\log(\log v+3)$ then go to Step 3d.
\item
If $p\notin S$ and $p$ is prime then set $S\leftarrow S\cup \{p\}$ and $b\leftarrow b+\log p$.
\item
If $b> B_l+2\log 2$ then output $S$ and terminate.
\end{enumerate}
\algitem
Set $x\leftarrow 2x$ and go to Step 2.
\end{enumerate}
\end{algorithm}

In Step~3a, the integer $t$ is generated as $t=al+2$ using a uniformly random integer $a\in[0,V/l-2]$.
The computation of $\omega(v)$ in Step 3b is performed by factoring~$v$.
Note that $\# S\le n$, and $p \le x$ for all $p\in S$.

\renewcommand\labelenumi{(\theenumi)}
\begin{lemma}\label{selectprimes}
Let $\mathcal{F}$ be a suitable family of orders, let $l$ be an odd prime and let $D=\disc(\mathcal{F}(l))$.
Given inputs $l$ and $D$, the expected running time of Algorithm~\ref{alg21} is finite.
Under the GRH we also have the following:
\begin{enumerate}
\item
The expected running time is $O(l^{1+\varepsilon})$, for any $\varepsilon\in\R_{>0}$.
\item
There is a constant $c<1$ such that for all $l>7$ and $k\in\Z_{>0}$, the algorithm terminates with $\log x \le (6+k)\log l$ with probability at least $1-c^{-k\log l}$.
\end{enumerate}
\end{lemma}
\begin{proof}
To analyze Algorithm~\ref{alg21}, we count the number of times Step 4 is executed, referring to the period between each execution as an iteration.
By Theorem~\ref{smallprimes}, the set of primes that satisfy Theorem~\ref{theprimes}, equivalently, those that satisfy (\ref{normeq2}), has positive density.  Here we may use the natural density, via~\cite[Thm.~4.3.e]{Jarden:Density}.
For every fixed odd prime $l$, this implies a lower bound of $\Omega(1/\log x)$ on the probability that a random integer in $[1,x]$ is a prime that satisfies (\ref{normeq2}).
Each integer~$p$ tested by Algorithm~\ref{alg21} necessarily satisfies (\ref{normeq2}), hence such a~$p$ is prime with probability $\Omega(1/\log x)$.
By Lemma~\ref{HWbound} in the appendix, the probability that a candidate~$p$ is skipped due to the test in Step 3b is $o(1)$.
This implies that for all sufficiently large~$x$, the probability that 
Algorithm~\ref{alg21} terminates in a given iteration is bounded above zero, and the expected running time is finite.

Now assume the GRH and let $l>7$.
Applying Theorem~\ref{smallprimes} with $c_0=1$, there are at least $l^3(\log l)^3$ primes $p \le c_1l^6(\log l)^4$ that satisfy (\ref{normeq2}), for some constant $c_1\in\R_{>0}$ that does not depend on $l$.
Let $x_0$ be the least value of $x\ge c_1l^6(\log l)^4)$.
When $x=x_0$ we have $VT/l \le 8c_1l^3(\log l)^4$, since $|D|\ge l^2$, and the probability that a given primality test succeeds is at least $8c_1/\log l \ge c_2/\log x$, for some constant $c_2\in\R_{>0}$.
From the inequality (\ref{LCbound}) in the proof of Theorem~\ref{smallprimes}, one finds that this holds for all $x \ge x_0$, with the same constants.
As above, Step 3b has negligible impact, and for $x\ge x_0$ the probability that the algorithm terminates in a given iteration is at least $c$, for some constant $c\in\R_{>0}$ independent of $l$ and $x$.

We now consider the running time as a function of $l$, fixing an arbitrary $\varepsilon\in\R_{>0}$.
It takes $O(\log l)$ iterations to achieve $x=x_0$, assuming that we don't terminate earlier, and we execute Steps 3b and 3c a total of $O(l(\log l)^2)$ times during this process.
The computation of $\omega(v)$ and the primality test of $p$ can both be achieved in expected time subexponential in $\log x$, by~\cite{Lenstra:RigorousFactoring}, yielding an $O(l^{1+\varepsilon})$ bound on the time to reach $x=x_0$, since $\log x_0 = O(\log(l))$.

For $x\ge x_0$, the probability of reaching each subsequent iteration declines exponentially, while the cost of Steps 3b and 3c grows subexponentially, implying that the total expected running time is also $O(l^{1+\varepsilon})$, proving (1).

Claim (2) follows from the same analysis.  We have $\log x_0 = (6+o(1))\log l$ and add $\log 2$ to $\log x$ in each iteration.  Once $x=x_0$, it takes more than $k\log l$ iterations to reach $\log x > (6+k)\log l$.  There is a probability of at least $c$ that the algorithm terminates in each subsequent iteration, yielding the bound in (2).
\end{proof}

\subsection{CRT computations}\label{CRT}
The computations involved in Steps~3, 4b, and 5 of Algorithm~\ref{alg2} are described in detail in~\cite[\S6]{Sutherland:HilbertClassPolynomials}.
We summarize briefly here.

Given $S=\{p_i\}$, let $M=\prod p_i$, $M_i=M/p_i$, and $a_i\equiv M_i^{-1}\bmod p_i$.
Let $c$ denote a coefficient of $\Phi_l\in\Z[X,Y]$, and let $c_i\equiv c\bmod p_i$ denote the corresponding coefficient of $\Phi_l\in\mathbf{F}_{p_i}[X,Y]$.  As in~\cite[\S10.3]{Gathen:ComputerAlgebra}, we can use fast Chinese remaindering to efficiently compute
\begin{equation}\label{standardCRTsum}
c\equiv c_ia_iM_i\bmod M.
\end{equation}
Provided that $M>2|c|$, we can then lift the result from $\Z/M\Z$ to $\Z$.

When $m$ is ``large," by which we mean $m\ge M > 2|c|$, we compute $c\bmod M$, lift to $\Z$, and output the integer $c$ as its representative modulo $m$.  In this scenario Step 4b simply stores the coefficients $c_i$ and Step 3 can be deferred to Step 5.

When $m$ is ``small," by which we mean $\M(\log m) = O(\log^3 l\log\log l)$, we instead use the explicit CRT modulo $m$.
Assuming $M>4|c|$, we may apply
\begin{equation}\label{explicitCRTsum}
c\equiv c_ia_iM_i - rM \bmod m,
\end{equation}
where $r$ is the closest integer to $s=\sum c_ia_i/p_i$, by~\cite[Thm.~3.1]{Bernstein:ModularExponentiation}.
In this scenario, we update the sum $C=\sum c_ia_iM_i\bmod m$ and an approximation to $s$ in Step 4b as each $c_i$ is computed.
This uses $O(\log m + \log l)$ space per coefficient, rather than the $O(l\log l)$ space used to compute $c\bmod M$.
The postcomputation in Step 5 determines $r$ from the approximation to $s$ and computes $c\bmod m$ via (\ref{explicitCRTsum}).

When $m$ is neither small nor large, a hybrid approach is used, see~\cite[\S6.3]{Sutherland:HilbertClassPolynomials}.

\subsection{Computing $\Phi_l(X,Y)\bmod p$}\label{Alg1}
An overview of Algorithm~\ref{alg1} was given in the introduction, we now fill in the details.

\renewcommand\labelenumi{\theenumi.}
\renewcommand\labelenumii{\theenumii.}
\algstart{{\bf \ref{alg1}}} {Let $l$, $p$, and $\O$ be as in Theorem~\ref{theprimes}, with $h(\O)\ge l+2$, and let $R=\Z+l\O$.
Given $H_\O\in\Z[X]$, compute $\Phi_l\in\Fp[X,Y]$ as follows:}
\algitem
Compute the presentations $\vec{\alpha}$ of $\Pic(\O)$ and $\vec{\alpha}'$ of $\Pic(R)$ suitable for $p$.
\algitem
Find a root of $j_0$ of $H_\O(X)$ over $\Fp$.
\algitem
Use $\vec{\alpha}$ to enumerate $\Ell_\O(\Fp)$ from $j_0$ and identify the $l$-isogeny cycles.
\algitem
For distinct $j_0,\ldots,j_{l+1}\in\Ell_\O(\Fp)$:
\begin{enumerate}
\item
Construct a curve $E_i$ with $j(E_i)=j_i$ such that $l$ divides $\#E_i(\Fp)$.
\item
Generate a random point $P\in E_i(\Fp)$ of order $l$.
\item
Use $E_i$ and $P$ to compute an $l$-isogenous curve $E_i'/\Fp$ via Algorithm~\ref{alg11}.
\item
If $j(E_i')\not\in\Ell_\O(\Fp)$ then set $j_i'\leftarrow(E_i')$, otherwise return to Step 3b.
\end{enumerate}
\algitem
Use $\vec{\alpha}'$ to enumerate $\Ell_R(\Fp)$ from $j_0'$ and identify the $l^2$-isogeny cycles.
\algitem
For $i$ from 0 to $l+1$:
\begin{enumerate}
\item
Let $j_{i0},\ldots,j_{il}$ consist of the neighbors of $j_i$ in its $l$-isogeny cycle in $\Ell_\O(\Fp)$ together with the $l^2$-isogeny cycle of $\Ell_R(\Fp)$ containing $j_i'$.
\item
Compute $\Phi_l(X,j_i)=\sum_k a_{ik}X^k$ as the product $\prod_k(X-j_{ik})$.
\end{enumerate}
\item
For $k$ from 0 to $l+1$:
\begin{enumerate}
\item
Interpolate $\phi_k\in\Fp[Y]$ with $\deg\phi_k\le l+1$ satisfying $\phi_k(j_i)=a_{ik}$.
\end{enumerate}
\algitem
Output $\Phi_l(X,Y)=\sum_k\phi_k(Y)X^k$.
\algend
\noindent
Steps 2, 6, and 7 involve standard computations with polynomials over finite fields, as described in~\cite{Gathen:ComputerAlgebra}, for example.
Step 1 is addressed in Section~\ref{polycyclic}, and Steps 3 and 5 are the topic of Section~\ref{CMaction2}.
Only Step 4 merits further discussion here.

The existence of the curve $E_i$ constructed in Step 4a is guaranteed by Theorem~\ref{theprimes}.
The trace of Frobenius $t$ of the desired curve is uniquely determined by the norm equation for $p$ and the constraint $t\equiv 2\bmod l$, as in (\ref{normeq2}).
With $k=j_i/(1728-j_i)$, the curve $E/\Fp$ defined by $y^2=x^3+3kx+2k$ has $j(E)=j_i$, and we may determine whether it is $E$ or its quadratic twist that has trace $t$ by attempting to generate a point of order $l$ on both curves in Step 4b.

To obtain a point $P$ of order $l$, we generate a random point $Q$ uniformly distributed over $E_i(\Fp)$,
compute the scalar multiple $P=nQ$, where $n=(p+1-t)/l$, and then check that $P\ne 0$.
This will be true with probability $1-1/l^2$, since Theorem~\ref{theprimes} implies that the Sylow $l$-subgroup of $E_i(\Fp)$ is $E_i[l]\cong\Z/l\Z\times\Z/l\Z$.  Thus we expect to succeed within $1+O(1/l^2)$ attempts, and the expected cost of generating $P$ is $O(\log p)$ operations in $\Fp$.

Note that $E_i(\Fp)$ contains $l+1$ distinct subgroups of order $l$, each corresponding to a distinct $l$-isogenous $j$-invariant.
At most 2 of these lie in $\Ell_\O(\Fp)$.
Thus in Step 4d we have $j(E_i')\not\in\Ell_\O(\Fp)$ with probability at least $1-2/(l+1)$ and expect to need $1+O(1/l)$ random points $P$ to obtain such an $E_i'$.
The curve $E_i'$ is the image of an $l$-isogeny whose kernel is generated by $P$, obtained via Algorithm~\ref{alg11}.

\subsection{Isogenies from subgroups}
Let $E/\Fp$ be an elliptic curve.
Given a cyclic subgroup $H\subseteq E(\barFp)$, V\'{e}lu's formulas construct an isogeny $E\to E'$ with $H$ as its kernel~\cite{Velu:Isogenies}.
In our setting $H$ is actually generated by an $\Fp$-rational $l$-torsion point, allowing us to work in $\Fp$ rather than an extension field.
Additionally, the order $l$ of $H$ is odd and $p>3$, allowing us to simplify the formulas.

\renewcommand\labelenumi{\theenumi.}
\renewcommand\labelenumii{\theenumii.}
\begin{algorithm}\label{alg11}
{Let $l>2$ and $p>3$ be primes, let $E/\Fp$ be an elliptic curve defined by $y^2=x^3+Ax+B$, and let $P=(P_x,P_y)$ be a point on $E(\Fp)$ of order $l$.  Compute the image $E'/\Fp$ of the $l$-isogeny with kernel $H=\langle P\rangle$ as follows:}
\begin{enumerate}
\algitem
Set $t\leftarrow 0$, $w\leftarrow 0$, and $Q\leftarrow P$.
\algitem
Repeat $(l-1)/2$ times:
\begin{enumerate}
\item
Set $s\leftarrow 6Q_x^2+2A$, and then set $u\leftarrow 4Q_y^2 + sQ_x$.
\item
Set $t\leftarrow t + s$, $w\leftarrow w + u$, and $Q\leftarrow Q + P$.
\end{enumerate}
\algitem
Set $A'=A-5t$ and $B'=B-7w$.
\algitem
Output the curve $E'/\Fp$ defined by $y^2=x^3+A'x+B'$.
\end{enumerate}
\end{algorithm}

The addition $Q+P$ in Step 2b is performed using the group operation in $E(\Fp)$.
The complexity of Algorithm~\ref{alg11} is $O(l)$ operations in $\Fp$.

\subsection{Complexity analysis}
We first bound the complexity of Algorithm~\ref{alg1}, as used by Algorithm~\ref{alg2}.

\begin{lemma}\label{Alg1bound}
Let $\mathcal{F}$ be a suitable family of orders that satisfies the condition of Theorem~\ref{suitablepresentation}.
For an odd prime $l$, let $\O=\mathcal{F}(l)$ and let $D=\disc(\O)$.
Let $p$ be a prime in the set $S$ selected by Algorithm~\ref{alg21} on input $l$ and $D$.  Assuming the GRH, the expected running time of Algorithm~\ref{alg1} is $O(l^2(\log p)^3\log\log p)$.
\end{lemma}
\begin{proof}
We note that $p=t^2-v^2l^2D > l^4$, thus $\log l < \log p$, 
and recall that the bit-complexity of multiplying two polynomials of degree $O(l)$ in $\Fp[X]$ may by bounded by $O(\M(l\log p))$, using Kronecker substitution, see Corollaries 8.28 and 9.8 of~\cite{Gathen:ComputerAlgebra}.

In the analysis below we use $O(\M(l\log p))=O(l(\log p)^2\log\log p)$, via the bound $\M(n)=O(n\log n\log\log n)$ for fast multiplication~\cite{Schonhage:Multiplication}.
When computing the cost of multiplications in $\Fp$ we use the weaker bound $\M(\log p)=O((\log p)^2/(\log \log p)^c)$, where $c$ is any constant, which is more convenient and does not change the overall bound.
We bound the cost of inversions in $\Fp$ by $O(\M(\log p)\log\log p)=O((\log p)^2)$, via~\cite[Cor.~11.10]{Gathen:ComputerAlgebra}.
\smallbreak

We now bound the (expected) cost of each step in Algorithm~\ref{alg1}:
\noindent
\begin{enumerate}
\item
We have $h(\O)<h(R)=O(l^2)$.
As described in Section~\ref{polycyclic}, the cost of computing the presentations $\vec{\alpha}$ and $\vec{\alpha}'$ is $O(l^2)$ operations in $\Pic(\O)$ and $\Pic(R)$.
Using binary quadratic forms, each group operation has complexity $O((\log l)^2)$, by~\cite{Biehl:FormReductionComplexity}, yielding an $O(l^2(\log l)^2)=O(l^2(\log p)^2)$ bound on Step~1.
\item
Using Berlekamp's probabilistic root-finding algorithm~\cite[\S7]{Berlekamp:PolyFactoringLargeFF} with a fast GCD computation~\cite[Alg.~11.4]{Gathen:ComputerAlgebra}, the expected time to find a single root of $H_D\in\Fp[X]$ may be bounded by $O(\M(l)(\log l + \log p))$ operations in $\Fp$, since $\deg H_D = O(l)$.  This implies an $O(l(\log p)^3)$ bound on Step~2.
\item
For $p\in S$ we have $\omega(v)\le 2\log\log p$, yielding an $O(\log\log p \log\log\log p)$ bound on $\max l(\vec{\alpha})$, by Theorem~\ref{suitablepresentation}.
From Lemma~\ref{CMactioncost}, $O((\log\log p)^2\log p)$ operations in $\Fp$ suffice to compute the action of any element of $\vec{\alpha}$.
We obtain an $O(l(\log p)^3)$ bound on the time to enumerate $\Ell_\O(\Fp)$, which dominates the $O(l\log l)$ time to identify the $l$-isogeny cycles.
\item
From the discussion in Section~\ref{Alg1} and the complexity of Algorithm~\ref{alg11}, we expect to use $O(l^2+l\log p)$ operations in $\Fp$ during Step~4.  This yields an $O(l^2(\log p)^2+l(\log p)^3)$ bound.
\item
Recall that the surjective map $\varphi:\Pic(R)\to\Pic(\O)$ in (\ref{exactsequence}) preserves the norms of representative ideals.
The subgroup of $\Pic(R)$ generated by invertible $R$-ideals with norms in $l(\vec{\alpha})$ contains $\varphi^{-1}(\Pic(\O))$.
It follows that the elements of~$\vec{\alpha}'$ with norm at most $\max l(\vec{\alpha})$ generate a subgroup of size at least $h(\O)>l$.
All but $O(l)$ of the $O(l^2)$ steps taken when enumerating $\Ell_R(\Fp)$ involve these elements, and, as in Step~3, we obtain a total cost of $O(l^2(\log p)^3)$ for these steps.
Assuming the GRH, the remaining elements of~$\vec{\alpha}'$ all have norm $O((\log |D|)^2)=O((\log l)^2)$, by~\cite{Bach:ERHbounds}, yielding a total cost of $O(l(\log l)^4(\log p)^3)$ for these steps, via Lemma~\ref{CMactioncost}.
Thus the expected time to enumerate $\Ell_R(\Fp)$ is $O(l^2(\log p)^3)$, which dominates the $O(l^2\log l)$ time to identify the $l^2$-isogeny cycles.
\item
Using a product tree we may compute $\prod_k(X-j_{ik})$ in time $O(\M(l\log p)\log l)$, yielding a total cost of $O(l^2(\log p)^3\log\log p)$ for Step~6.
\item
Using a product tree and fast interpolation~\cite[Alg.~10.11]{Gathen:ComputerAlgebra}, we also obtain a cost of $O(l^2(\log p)^3\log\log p)$ for Step~7.  Here we use the $O(\M(l\log p))$ bit-complexity of polynomial multiplication in $\Fp[X]$ to bound the cost at each level, rather than using the bound in~\cite[Cor.~10.12]{Gathen:ComputerAlgebra}.
\end{enumerate}
The bound $O(l^2(\log p)^3\log\log p)$ applies to every step, completing the proof.
\end{proof}
\noindent
We are now ready to prove our main theorem, which bounds the complexity of using Algorithm~\ref{alg2} to compute $\Phi_l\bmod m$, where $l$ is an odd prime and $m$ is any positive integer.
Recall that that the algorithm must be given a suitable family of orders $\mathcal{F}$, as defined in Definition \ref{suitableorders}, and to prove our complexity bound we additionally require that $\mathcal{F}$ satisfy the property given in Theorem~\ref{suitablepresentation}.
Example~\ref{exam} provides one such $\mathcal{F}$, and there are many others that can be efficiently computed and may yield better performance, as discussed in Section~\ref{selectorder}.

\begin{thm1}
Let $\mathcal{F}$ be a suitable family of orders that satisfies the condition of Theorem~\ref{suitablepresentation}.
Let $l$ be an odd prime and let $m\in\Z_{>0}$.
Given inputs $l$, $m$, and $\O=\mathcal{F}(l)$, Algorithm~\ref{alg2} correctly computes $\Phi_l\in(\Zm)[X,Y]$.
Under the GRH, its expected running time is
$$
O(l^3 \log ^3 l\log\log l),
$$
using $O(l^2\log lm)$ expected space.
\end{thm1}
\begin{proof}
We first argue correctness.
By Lemma~\ref{selectprimes}, Algorithm~\ref{alg21} obtains a set of primes~$S$ that satisfy Theorem~\ref{theprimes}, with $\prod_{p\in S} p > 4|c|$, for every coefficient $c$ of~$\Phi_l$.
We now claim that for $p\in S$, Algorithm~\ref{alg1} obtains, for each of $l+2$ distinct $j$-invariants $j_i$, a list of $l+1$ distinct $j$-invariants $j_{ik}$ of $l$-isogenous curves.
Granting the claim, we may invoke standard properties of $\Phi_l$ to show that Algorithm correctly interpolates $\Phi_l\in\Fp[X,Y]$,
see~\cite[Thm.~12.19]{Washington:EllipticCurves} and~\cite[Thm.~5.3]{Lang:EllipticFunctions}, for example.
The correctness of Algorithm~\ref{alg2}  then follows from the CRT and/or the explicit CRT mod $m$, via~\cite[Thm.~3.1]{Bernstein:ModularExponentiation}, as described in Section~\ref{CRT}.

As usual, let $\O=\mathcal{F}(l)$ have fraction field $K$, and let $R=\Z+l\O$.
The claim above rests on three facts: (1) the explicit CM-action described in Section~\ref{CMaction1} is correct, (2) any two $l^2$-isogenous elements of $\Ell_R(\Fp)$ must be $l$-isogenous to exactly one and the same element of $\Ell_\O(\Fp)$, and (3) each $l^2$ isogeny cycle in $R$ contains exactly $l-\inkron{d_K}{l}$ elements.
We note that (1) follows from the theory of complex multiplication and the properties of $\Phi_{l_0}$ guaranteed by~\cite[Thm.~12.19]{Washington:EllipticCurves}, (2) follows from the $l$-volcano structure, as shown by~\cite[\S2.2]{Fouquet:IsogenyVolcanoes} and~\cite[Prop.~23]{Kohel:thesis}, and (3) is explicitly proven in Lemma~\ref{cyclic}.
\smallbreak
We now assume the GRH and bound the complexity of Algorithm~\ref{alg2}.
Lemma~\ref{selectprimes} shows that the expected size of the largest $p\in S$ is $O(\log l)$, and we have $\#S=O(l)$.
Applying~\cite[\S6]{Sutherland:HilbertClassPolynomials} yields an $O(l^2\log lm)$ space bound for $m\in\Z_{>0}$.

By~\cite[Thm.~1]{Sutherland:HilbertClassPolynomials}, the expected time to compute $H_\O$ in Step 1 is $O(l^{2+\varepsilon})$,
and Lemma~\ref{selectprimes} gives an expected time of $O(l^{1+\varepsilon})$ for Step 2, for any $\varepsilon\in\R_{>0}$.
Additionally, we have $\log p > (6+k)\log l$ for all $p\in S$ with probability approaching 1 exponentially as $k$ increases.
The time complexity of all remaining steps in Algorithm~\ref{alg2}, including calls to Algorithm~\ref{alg1}, depends polynomially on $\log p$, hence we may bound the expected running time assuming $\log p=O(\log l)$.

Regardless of the exact cutoff used, if $\M(\log m)=O((\log l)^3\log\log l)$ whenever we consider $m$ ``small", we may apply the results of~\cite[\S6]{Sutherland:HilbertClassPolynomials} to obtain a bound of $O(l^3(\log l)^3\log\log l)$ on the expected time for all CRT computations, for every $m\in\Z_{>0}$.
Since $\mathcal{F}$ satisfies the property of Theorem~\ref{suitablepresentation}, we may apply Lemma~\ref{Alg1bound} with $\log p=O(\log l)$ to obtain an $O(l^2(\log l)^3\log\log l)$ bound on the expected time of each call to Algorithm~\ref{alg1}.  Applying $\#S=O(l)$ completes the proof.
\end{proof}

\subsection{Selecting a suitable order}\label{selectorder}
The family of orders used in Theorem~\ref{main-thm} suffices to prove the complexity bound, but we can simplify the implementation and improve performance with some additional constraints on the order~$\O$.  Let us fix a bound $b < l$ (say $b=256$, for large $l$), and a small prime $l_0 < l$ (typically $l_0=2$).
As above, $R$ is the order of index $l$ in $\O$, and $\O_K$ is the maximal order.
We seek an order $\O$ for which the following hold:
\renewcommand\labelenumi{(\theenumi)}
\begin{enumerate}
\item
The conductor of $\O$ is $b$-smooth, $h(\O_K) \le b$, and $h(\O) \ge l+2$.
\item
The groups $\Pic(\O)$ and $\Pic(R)$ are either generated by a single ideal with norm $l_0$, or by two ideals with norms $l_0$ and $l_1$, where $l_1\le b$ is ramified.
\end{enumerate}
The first condition ensures that $\O$ is suitable for $l$ and allows us to to obtain a root of $H_\O(X)$ using only polynomials of degree at most~$b$.  This is accomplished by finding a root of $H_{\O_K}(X)$ and descending to the proper level of the $l'$-isogeny volcano for each prime $l'\le b$ dividing the conductor of $\O$, as in~\cite[\S4.1]{Sutherland:HilbertClassPolynomials}.
The second condition allows us to realize the torsors for $\Pic(\O)$ and $\Pic(R)$ either by walking a single $l_0$-isogeny cycle, or by walking two $l_0$-isogeny cycles connected by a single $l_1$-isogeny.  In the latter case we orient the two cycles by computing one extra $l_1$-isogeny.  In both cases we avoid the need to ever distinguish the action of an ideal and its inverse, simplifying the computation described in Section~\ref{CMaction2}.

Subject to these conditions, we also wish to minimize $h(\O)\ge l+2$.
To find such orders we enumerate fundamental discriminants $d_K<-4$ with $\inkron{d_K}{l_1}=1$ and $h(d_K)\le b$, and for each $d_K$ we select $b$-smooth integers $u$ for which $h(u^2d_K)$ is slightly greater than $l+2$ and test whether condition (2) holds.
In practice we are almost always able to obtain $\O$ with $h(\O)$ within a few percent of $l+2$.

\section{Modular functions other than $j$}\label{OtherFunctions}
Let $g$ be a modular function of level $N$, and let $l\nmid N$ be a prime. 
We define the {\it modular polynomial 
$\Phi_l^g$ of level~$l$ for $g$\/} as the minimal polynomial of the function
$g(lz)$ over the field $\CC(g)$. Much of the theory for the classical modular 
polynomial of the $j$-function generalizes to $g$. In particular, if the
Fourier expansion of~$g$ has {\it integer\/} coefficients then we have
$\Phi_l^g \in \Z(g)[X]$. The following lemma gives us further information
in this case.

\begin{lemma} Let $g$ be a modular function and let $l$ be a prime not 
dividing the level of~$g$. Suppose that $\Phi_l^g$ has integer coefficients.
 If $g$ is invariant under the action of either
$S = \bigl(\afrac{0}{1}\thinspace\afrac{-1}{0}\bigr)\in\SL_2(\Z)$ or
$M = \bigl(\afrac{0}{1} \thinspace\afrac{-l}{0}\bigr)\in\GL_2(\Q)$, then
we have
$$
\Phi_l^g(X,Y) = \Phi_l^g(Y,X).
$$
\end{lemma}
\begin{proof}
The proof follows the symmetry proof for $\Phi_l(X,Y)$, see ~\cite[Thm.~5.3]{Lang:EllipticFunctions}.
\end{proof}

The polynomial $\Phi_l^g$ should not be confused with the minimal polynomial of $g$ as an element of $\CC(j)$, which we
denote $\Psi^g(X,J)$.  The polynomial $\Psi^g$ depends only on $g$, not $l$,
and we assume it is known (for our purposes, it effectively defines~$g$).  Given $\Psi^g$, our goal is to efficiently compute $\Phi_l^g$ for a prime $l\nmid N$.  

To apply our method we require that $\Phi_l^g$ have degree $l+1$ (in both $X$ and $Y$).
The degree of $\Phi_l^g$ can be explicitly computed using \cite[\S 5.2]{Broker:pAdicClassInvariants}, and we note that this degree must be $l+1$ when $\deg_J \Psi^g = 1$ (this applies to the function $\gamma_2$ and the Weber $\mathfrak{f}$-function considered in Sections \ref{ComputingGamma2} and \ref{Weberf}), and also when $\deg_J \Psi^g = 2$ and $g$ is invariant under the Atkin-Lehner involution (this applies to the various modular functions considered in Section \ref{miscfunctions}).

We wish to adapt Algorithm~\ref{alg1} to compute $\Phi_l^g \in \Fp[X]$.  We may then apply Algorithm~\ref{alg2}
to recover $\Phi_l^g$ over the integers or modulo some integer~$m$ via the Chinese Remainder Theorem.
To simplify matters, we place some additional restrictions on the order~$\O$ that we
use in Algorithm~\ref{alg2}.
Specifically, we require that there is a generator $\tau\in\HH$ of $\O=\Z[\tau]$ with the property that 
$$
g(\tau) \in K_\O,
$$
where $K_\O$ is the ring class field for the order~$\O$. We say that $g$
is a {\it class invariant\/} for~$\O$ in this case.  If we now take a prime $p$ that splits
completely in $K_\O$ and $E/\Fp$ an elliptic curve with endomorphism ring~$\O$, then the polynomial
$$
\Psi^g(X, j(E)) \in \Fp[X]
$$
has at least one root in~$\Fp$. Indeed, the value $h=g(\tau) \bmod \gp$, for
a prime $\gp|p$ of~$K_\O$, satisfies $\Psi^g(h,j(E)) = 0$.

We can analyze {\it how many roots\/} the polynomial $\Psi^g(X,j(E))$ has in $\Fp$
using a combination of Deuring lifting and Shimura reciprocity. We refer
to~\cite[\S6.7]{Broker:Thesis} for a detailed description of the 
techniques involved and only state the result here. 
Let $g_i: \HH \rightarrow \C$ be the roots of $\Psi^g(X,j)$. The 
functions $g_i$ are modular of level~$N$, and they are permuted by the 
Galois group $\GL_2(\Z/N\Z)$ of the field of all modular functions of 
level~$N$. To state our result, we will associate a matrix $A \in 
\GL_2(\Z/N\Z)$ to the Frobenius morphism of~$E$ as follows. 
Fix an isomorphism $\End(E) \isar \O$ and let $\pi_p\in\O$ be
the image of the Frobenius morphism. If $\pi_p$ has minimal 
polynomial $X^2-tX+p$ of discriminant $\Delta = t^2-4p$, then we put
\begin{equation}\label{Amatrix}
A = \left( \begin{matrix} 
\frac{t-2}{2} & \frac{\Delta - 1}{2} \\
2 & \frac{t+2}{2} 
\end{matrix} 
\right) \in \GL_2(\Z/N\Z).
\end{equation}

\begin{theorem} \label{Shimura}
{Let $g$ be a modular function with the property that 
$\Psi^g$ has integer coefficients and is separable modulo a prime~$p$ that
does not divide the level of~$N$. If $E/\Fp$ is an elliptic curve with 
endomorphism ring~$\O$ and $j_0=j(E)$, then we have
$$
\# \{ x \in\Fp : \Psi^g(x,j_0) = 0 \} = \# \{ g_i : \Psi^g(g_i,j) =0 
\hbox{\ and \ } g_i^A = g_i \},
$$
where $A$ is the matrix in $(\ref{Amatrix})$.}
\end{theorem}
\proof See~\cite[\S6.7]{Broker:Thesis}. \endproof

To check if $g_i^A = g_i$ holds is a standard computation, see~\cite{Gee:GeneratingClassFields} for example.
Although the matrix $A$ has norm~$p$
and trace~$t$ and therefore depends on~$p$, we can often derive a result 
that merely depends on a congruence condition on $p \bmod N$. Lemma~\ref{WeberTwoRoots}
in Section~\ref{Weberf} gives an example.

\subsection{Computing modular polynomials for $\boldsymbol{\gamma_2}$}\label{ComputingGamma2}
Let $\gamma_2(z)$ denote the unique cube root of $j(z)$ that has integral
Fourier expansion.
It was known to Weber already that $\gamma_2$ is a modular function of 
level~3, see~\cite[\S125]{Weber:Algebra}, and $\gamma_2$ is a class
invariant for~$\O$ whenever $3\nmid\disc(\O)$.  We have 
$\Psi^{\gamma_2}(X,j) = X^3-j$, and in this simple case there is no need to
apply Theorem~\ref{Shimura}; if we restrict to $p = 2 \bmod 3$ then every element of $\Fp$
has a unique cube root. 
Thus we can compute $\Phi_l^{\gamma_2}$ with only minor modifications to Algorithm~\ref{alg2}:
\renewcommand\labelenumi{\theenumi.}
\begin{itemize}
\item
Use a suitable order $\O$ with $3\nmid \disc(\O)$ and select only primes $p\equiv 2\bmod 3$.
\item
After Step 5 of Algorithm~\ref{alg1}, replace each element of $\Ell_\O(\Fp)$ 
and $\Ell_R(\Fp)$ with its unique cube root in $\Fp$.
\end{itemize}
\noindent
These changes suffice, but we can also improve the algorithm's performance.

First, Lemmas 2--3 and Corollary 9 of~\cite{BrokerSutherland:PhiHeightBound} yield the bound
\begin{equation}\label{gamma2height}
B_l^{\gamma_2} = 2l\log l + 8l
\end{equation}
on the logarithmic height of $\Phi_l^{\gamma_2}$ (conjecturally, $B_l^{\gamma_2}=2l\log l +4l$ for all $l>60$, but we do not use this).
Thus we can reduce the height bound $B_l$ in (\ref{heightbound}) by a factor of  approximately 3 when computing $\Phi_l^{\gamma_2}$.
This reduces the number of primes $p\in S$, and the corresponding number of calls Algorithm~\ref{alg2} makes to Algorithm~\ref{alg1}.

Second, we may take advantage of the fact that $\Phi_l^{\gamma_2}$ is {\it 
sparser\/} than $\Phi_l$.
As noted in~\cite[p.~37]{Elkies:AtkinBirthday}, the coefficient of $X^aY^b$ in $\Phi_l^{\gamma_2}$ is zero unless
\begin{equation}\label{gamma2sparse}
a + lb \equiv l+1 \mod 3.
\end{equation}
The proof of this relation goes back to Weber: the argument given 
in~\cite[p.\ 266]{Weber:Algebra} generalizes immediately to~$\gamma_2$.
It allows us to reduce the number 
of points we use to interpolate $\Phi_l^{\gamma_2}$ by a factor of 
approximately 3.

Let $n=\lceil(l+1)/3\rceil+1$.  When selecting a suitable order $\O$ as in Section~\ref{selectorder}, we now only require that $h(\O)\ge n$, and further modify Algorithm~\ref{alg1} as follows:
\begin{itemize}
\item
In Steps 4-6 we construct just $n$ polynomials $\Phi_l^{\gamma_2}(X,\sqrt[3]{j_i})$ of degree $l+1$.
\item
In Step 7 we interpolate $l+1$ polynomials $\phi^*_k$ of degree less than $n$ by writing $\phi_k=Y^c\phi^*_k(Y^3)$, with $c\in\{0,1,2\}$ satisfying $c+lk\equiv l+1\mod 3$.
\end{itemize}
This reduces the cost of all the significant components of Algorithm~\ref{alg1} by a factor of approximately 3.  
The reduction in the cost of the interpolations in Step~7 is actually greater than this, since its complexity is superlinear in the degree.

The total size of $\Phi_l^{\gamma_2}$ is approximately 9 times smaller than $\Phi_l$, and with the optimizations above, the time to compute it is effectively reduced by the same factor.  A small amount of additional time is required to compute the cube roots of the elements in $\Ell_\O(\Fp)$ and $\Ell_R(\Fp)$, but even this can be avoided.

Provided we have already computed $\Phi_{l'}^{\gamma_2}$ for some small values of $l'$ (specifically, for the primes $l_0$ and $l_1$ of Section~\ref{selectorder}), we may use these polynomials to directly enumerate sets $\Ell_\O^{\gamma_2}(\Fp)$ and $\Ell_R^{\gamma_2}(\Fp)$ containing the cube roots of the elements in $\Ell_\O(\Fp)$ and $\Ell_R(\Fp)$ respectively.  We need only compute the cube roots of $j_0$ and $j_0'$ as starting points.  This third optimization yields a small but useful improvement in the case of $\gamma_2$, and plays a critical role in the examples that follow.

\subsection{Recovering $\Phi_l$ from $\Phi_l^{\gamma_2}$}

Having computed $\Phi_l^{\gamma_2}$, we note that $\Phi_l$ may be computed via~\cite[Eq. 23]{Elkies:AtkinBirthday}:
\begin{equation}\label{gammaToj}
\Phi_l(X^3,Y^3) = \Phi_l^{\gamma_2}(X,Y)\Phi_l^{\gamma_2}(X,\omega Y)\Phi_l^{\gamma_2}(X,\omega^2 Y),
\end{equation}
where $\omega=e^{2\pi i/3}$.  For computation in $\Z$, or modulo $m$, it is more convenient to express $\Phi_l^{\gamma_2}$ in terms of polynomials $P_0,P_1,P_2\in\Z[X,Y]$ satisfying
\begin{equation}\label{gammaSplit}
\Phi_l^{\gamma_2}(X,Y) = P_0(X^3,Y^3)Y^b + P_1(X^3,Y^3)XY + P_2(X^3,Y^3)X^2Y^{2-b},
\end{equation}
where $b=2$ when $l\equiv 1\bmod 3$, and $b=0$ when $l\equiv 2\bmod 3$.  We then have
\begin{equation}\label{GammaToj}
\Phi_l = P_0^3 Y^b + (P_1^3 - 3P_0P_1P_2) XY + P_2^3 X^2 Y^{2-b}.
\end{equation}
Using Kronecker substitution and fast multiplication, it is possible to evaluate (\ref{GammaToj}) in time $O(l^3(\log l)^{2+\varepsilon})$, which is asymptotically faster than Algorithm~\ref{alg2}.  This suggests that we might more efficiently compute $\Phi_l$ by recovering it from $\Phi_l^{\gamma_2}$, but we do not find this to be true in practice: it actually takes longer to evaluate (\ref{GammaToj}) than it does to compute $\Phi_l$ directly.  This can be explained by two factors.  First, the $\Fp$-operations used in Algorithm~\ref{alg1} effectively have unit cost for word-size primes, making it faster than Theorem~\ref{main-thm} would suggest for all but very large $l$.  Secondly, the evaluation of (\ref{GammaToj}) becomes extremely memory intensive when $l$ is large.  However, if we are  computing $\Phi_l \bmod m$ with $\log m \ll l\log l$, then the time to apply (\ref{GammaToj}) modulo $m$ is negligible.  In this situation it is quite advantageous to derive $\Phi_l\bmod m$ from $\Phi_l^{\gamma_2}\bmod m$, as may be seen in Table~3 of Section~\ref{results}.

\subsection{Computing modular polynomials for the Weber $\mathfrak{f}$ function}\label{Weberf}
We now consider the classical Weber function~\cite[p.~114]{Weber:Algebra} defined by
$$
\mathfrak{f}(z)=\zeta_{48}^{-1}\frac{\eta((z+1)/2))}{\eta(z)},
$$
where $\zeta_{48}=e^{\frac{\pi i}{24}}$ and $\eta(z)$ is the Dedekind eta 
function.  This is a modular function of level 48 that satisfies 
$\gamma_2 = (\mathfrak{f}^{24}-16)/\mathfrak{f}^8$, see~\cite[p.~179]{Weber:Algebra}, thus we have
$$
\Psi^{\mathfrak{f}}(X,j) = (X^{24}-16)^3 - X^{24}j.
$$
Asymptotically, we expect to be able to reduce the height bound $B_l$ by a 
factor of $\deg_X\Psi^{\mathfrak{f}}/\deg_j\Psi^{\mathfrak{f}}=72$ when 
computing $\Phi_l^\mathfrak{f}$.  One can derive an explicit bound along the 
lines of~(\ref{gamma2height}), but this tends to overestimate the $O(l)$ term 
quite significantly, so in practice for large $l$ we use the heuristic bound
\begin{equation}\label{Weberheight}
B_l^\mathfrak{f} = \frac{1}{12}l\log l + \frac{1}{5}l\qquad\qquad(l>2400),
\end{equation}
which has been verified for every prime $l$ between 2400 and 10000.  The 
modular polynomial $\Phi_l^{\mathfrak{f}}$ is also sparse: the coefficient 
of $X^aY^b$ can be nonzero only when
\begin{equation}\label{Webersparse}
la+b\equiv l+1\bmod 24,
\end{equation}
as shown in~\cite[p.~266]{Weber:Algebra}.  Thus $\Phi_l^\mathfrak{f}$ is
roughly $72\cdot24=1728$ times smaller than $\Phi_l$.  By applying the
technique described above for $\gamma_2$, \emph{mutatis mutandis}, we can
actually compute $\Phi_l^\mathfrak{f}$ more than 1728 times faster than
$\Phi_l$ for large values of $l$, as may be seen in Table~2 of
Section~\ref{results}.  When applying Algorithm~\ref{alg2}, we now insist that the
order $\O$ have discriminant $D\equiv 1\bmod 8$ and $3\nmid D$, since
the Weber function yields class invariants in (at least) this case, see
~\cite{Gee:GeneratingClassFields}, for example.

Since $-\mathfrak{f}$ will also yield class invariants for~$\O$, the 
polynomial $\Psi^\mathfrak{f}(X,j_0)$ will always have at least {\it two\/}
roots in $\Fp$. The following lemma tells us that we can impose a congruence condition
on~$p$ to ensure that we have exactly two roots.

\begin{lemma}\label{WeberTwoRoots}
Let $p\equiv 11\bmod 12$ be prime and let $j_0$ be the $j$-invariant of an 
elliptic curve $E/\Fp$ with $\End(E)$ isomorphic to an imaginary quadratic 
order $\O$ with discriminant $D\equiv 1\bmod 8$ and $3\nmid D$.  
Then $\Psi^\mathfrak{f}(X,j_0)\in\Fp[X]$ has exactly two roots in $\Fp$, and 
these are of the form $x_0$ and $-x_0$.
\end{lemma}
\noindent
Note that if the lemma applies to $\O$, it also applies to the order $R$ of index $l$ in $\O$.

\begin{proof}
We only have to apply Theorem~\ref{Shimura}. The action of $A$ on the roots of 
$\Psi^\mathfrak{f}(X,j_0)$ is computed in~\cite[\S6.7]{Broker:Thesis}, and 
this yields the lemma.
\end{proof}

Given a $j$-invariant $j_0\in\Fp$ that corresponds to $j(\tau_0)\in K_\O$, we cannot readily determine which of the roots $x_0$ and $-x_0$ of $\Psi^\mathfrak{f}(X,j_0)$ actually corresponds to $\mathfrak{f}(\tau_0)$.  The functions $\mathfrak{f}$ and $-\mathfrak{f}$ yield distinct class invariants, but they share the same modular polynomials, since $\Phi_l^\mathfrak{f}(X,Y)=\Phi_l^\mathfrak{f}(-X,-Y) = \Phi_l^{-\mathfrak{f}}(X,Y)$, by (\ref{Webersparse}).

Thus for the initial $j_0$ obtained in Step 2 of Algorithm~\ref{alg1}, it does not matter whether we pick $x_0$ or $-x_0$ as a root of $\Psi^\mathfrak{f}(X,j_0)$, and we need not be concerned with making a consistent choice for each prime $p$.  However it is critical that while computing $\Phi_l^\mathfrak{f}\bmod p$ we make a consistent choice of sign for each $j$-invariant we convert to an ``$\mathfrak{f}$-invariant" (a root of $\Psi^\mathfrak{f}(X,j_i)\bmod p$).  This makes it impractical to enumerate $j$-invariants and convert them \emph{en masse}.  Instead, as described for $\gamma_2$ above, we use modular polynomials $\Phi_{l'}^\mathfrak{f}$ for small $l'$ to enumerate sets $\Ell_\O^{\mathfrak{f}}(\Fp)$ and $\Ell_R^{\mathfrak{f}}(\Fp)$ from starting points $x_0$ and $x_0'$ satisfying $\Psi^\mathfrak{f}(x_0,j_0)=\Psi^\mathfrak{f}(x_0',j_0')=0$.  This ensures that signs are chosen consistently within each of these sets; we only need to check that the sign choices for the two sets are consistent with each other.

To do so, we use the fact that the coefficient of $X^lY^l$ in $\Phi_l^\mathfrak{f}(X,Y)$ is $-1$.  This is shown for $\Phi_l$ in~\cite[\S69]{Weber:Algebra}, and the same argument applies to $\Phi_l^\mathfrak{f}$.  We modify Algorithm~\ref{alg1} to compute the coefficient of $X^lY^l$ in $\Phi_l^\mathfrak{f}\bmod p$ in between Steps 5 and 6.  We do this twice, switching the signs in $\Ell_\O^{\mathfrak{f}}(\Fp)$ the second time, and expect exactly one of these computations to yield $-1$, thereby determining a consistent choice of signs.  This test should be regarded as a heuristic, since we do not rule out the possibility that both choices produce $-1$.  However, in the course of extensive testing this has never happened, and we suspect that it cannot.  If it does occur, the algorithm can detect this and simply choose a different prime $p$.

\subsection{Eta quotients and Atkin modular functions}\label{miscfunctions}
For a prime $N$, let
$$
f_N(z) = N^{s/2}\left(\frac{\eta(Nz)}{\eta(z)}\right)^{s},
$$
where $s=24/\gcd(12,N-1)$.  These are modular functions of level $N$, and the polynomials $\Psi_N = \Psi^{f_N}$ that relate $f_N$ to $j$ are sometimes called canonical modular polynomials~\cite[p.~418]{Cohen:HECHECC}.  The functions $f_N$ are closely related to the functions $\mathfrak{w}_N^s$ considered in~\cite{Enge:GeneralizedWeberI}, and in fact $\Psi^{f_N}=\Psi^{\mathfrak{w}_N^s}$, so what follows applies to both. When $N$ is 2, 3, 5, 7, or 13, we have $\deg_j \Psi_N = 1$ and can adapt Algorithm~\ref{alg2} to compute polynomials $\Phi_l^{f_N}$ for odd primes $l\nmid N$.  We assume here that $N$ is also odd.

We have $\deg_X \Psi_N=N+1$, hence we can reduce the height bound $B_l$ by a
factor of approximately $N+1$.  When selecting a suitable order $\O$, we
require that $N$ is prime to the conductor and splits into prime ideals that
are distinct in $\Pic(\O)$. This assumption is stronger than we need, but it simplifies the implementation.
For the primes $p\in S$ we require that $N$ is prime to $v$, where $4p=t^2-v^2\disc(\O)$.

As shown in \cite{Muller:thesis}, the polynomial $\Psi_N(X,j_0)\bmod p$ has the same splitting type
as $\Phi_N(X,j_0)\bmod p$.  In particular, for $j_0\in\Ell_\O(\Fp)$ (or $j_0\in \Ell_R(\Fp)$) it has exactly two roots, say $x_1$, and $x_2$.  These correspond to $N$-isogenies as follows:
the $j$-invariants $j_1$ and $j_2$ of the two elliptic curves that are $N$-isogenous to $j_0$ are uniquely determined by the relations $\Psi_N(N^s/x_1,j_1)=0$ and $\Psi_N(N^s/x_2,j_2)$.  Here the transformation $x\mapsto N^s/x$ realizes the Atkin-Lehner involution on $f_N(z)$.

Starting points $x_0$ and $x_0'$ corresponding to $j_0$ and $j_0'$ are chosen as follows.  Let $x_0'$ be a root of $\Psi_N(X,j_0')$, chosen arbitrarily, and let $j_1'$ be determined by $\Psi_N(N^s/x_0',j_1')=0$.  We then use V\'elu's formulas to obtain the $j$-invariant $j_1$ of the  elliptic curve that is $l$-isogenous to $j_1'$ (there is exactly one and it lies in $\Ell_\O(\Fp)$, since $j_1'\in\Ell_R(\Fp)$ is on the floor of its $l$-volcano).  Finally, $x_0$ is uniquely determined by the constraints $\Psi_N(x_0,j_0)=0$ and $\Psi_N(N^s/x_0,j_1)=0$.

We next consider double eta-quotients~\cite{Enge:DoubleEtaQuotient,EngeSchertz:CompositeLevel} of
composite level $N=p_1p_2$:
$$
\mathfrak{w}_{p_1,p_2}^s(z) = \left(\frac{\eta(\frac{z}{p_1})\eta(\frac{z}{p_2})}{\eta(\frac{z}{p_1p_2})\eta(z)}\right)^s,
$$
where $p_1\ne p_2$ are primes and $s=24/\gcd(24,(p_1-1)(p_2-1))$.  For $(p_1,p_2)$ in
$$
\bigl\{(2,3),(2,5),(2,7),(2,13),(3,5),(3,7),(3,13),(5,7)\bigr\},
$$
the polynomial $\Psi_{p_1,p_2}=\Psi^{\mathfrak{w}_{p_1,p_2}^s}$ has degree 2 in
$j$ and we can compute $\Phi_l^{\mathfrak{w}_{p_1,p_2}^s}$ for odd primes
$l\nmid N$.  Our restrictions on $\O$ are analogous to those for $f_N$ or
$\mathfrak{w}_N^s$: we require that $N$ is prime to the conductor and that both
$p_1$ and $p_2$ split into distinct prime ideals in $\Pic(\O)$.  Our
requirements for $p\in S$ are as above.  We can reduce the height bound $B_l$ by a
factor of approximately $(p_1+1)(p_2+1)/2$.

With the double eta-quotients, the polynomial $\Psi_{p_1,p_2}(X,j_0)$ has four
roots, corresponding to four distinct isogenies of (composite) degree $N$.
Each root $x_i$ uniquely determines the $j$-invariant of a curve $N$-isogenous
to $E/\Fp$ as the unique root of $\Psi_{p_1,p_2}(x_i,J)/(J-j_0)\in\Fp[J]$.  The
double eta-quotients are invariant under the Atkin-Lehner involution, so we
need not transform $x_i$.  With this understanding, the procedure for selecting
$x_0$ and $x_0'$ is as above.

In some cases one can obtain smaller modular polynomials by considering
suitable roots of the functions defined above.  For example, a sixth root of
$f_3$ also yields class invariants (this is shown for $\mathfrak{w}_3^2$
in~\cite{Enge:GeneralizedWeberI}), and the corresponding modular polynomials
are sparser and of lower height (by a factor of 6).

Our algorithm also applies to the Atkin modular functions, which we denote
$A_N$.  These are (optimal) modular functions for $X_0^+(N)$ invariant under
the Atkin-Lehner involution, see
\cite{Elkies:AtkinBirthday,Morain:PointCounting} for further details.  The
polynomials $\Psi^{A_N}$ are known as Atkin modular polynomials, and are
available in computer algebra systems such as Magma~\cite{Magma} and Sage
\cite{SAGE}.  For primes $N<32$, and also $N$ in the set $\{41,47,59,71\}$, we
have $\deg_j\Psi^{A_N} = 2$ and can compute polynomials $\Phi_l^{A_N}$ for odd
primes $l\ne N$.  This is done in essentially the same way as with the double
eta-quotients, except that now $N$ is prime and $\Psi^{A_N}(X,j_0)$ has just
two roots, rather than four.  For these $A_N$, the height bound can be reduced
by a factor of approximately $(N+1)/2$.

Finally, we note an alternative approach applicable to both eta-quotients and
the Atkin modular functions.  If we choose $\O$ so that the prime factors of
$N$ are all ramified, then there is actually a unique $x_0\in\Fp$ corresponding
to each $j_0$ in $\Ell_\O(\Fp)$ and $\Ell_R(\Fp)$, that is, $\Psi^g(X,j_0)$ has exactly one root
in $\Fp$.  In this scenario we can
simply enumerate $j$-invariants as usual and then replace each $j_i$ with a
corresponding $x_i$.  This is not as
efficient and places stricter requirements on $\O$, but it allows us to compute
$\Phi_l^g$ without needing to know $\Phi_{l'}^g$ for any $l'$.  This provides a
convenient way to ``bootstrap" the process.  In fact all of the modular
polynomials $\Phi_l^g$ we have considered can eventually be obtained via
Algorithm~\ref{alg2}, starting from the polynomials $\Psi^g$ and $\Phi_2$.

\section{Computational results}\label{results}
We have applied our algorithm to compute polynomials $\Phi_l^g$ for all the modular functions discussed in Section~\ref{OtherFunctions} and every applicable $l$ up to 1000.  For the functions $j$, $\gamma_2$, and $\mathfrak{f}$ we have gone further, and present details of these computations here.

\subsection{Implementation} The algorithms described in this paper were implemented using the GNU C/C++ compiler~\cite{GNU} and the GMP library~\cite{GMP} on a 64-bit Linux platform.  Multiplication of large polynomials is handled by the 
{\tt zn\_poly} library developed by Harvey~\cite{Harvey:KroneckerSubstitution,Harvey:zn_poly}.

The hardware platform included four 3.0 GHz AMD Phenom II processors, each with four cores and 8GB of memory.  Up to 16 cores were used in the larger tests, with essentially linear speedup.  For consistency we report total CPU times, noting that in a multi-threaded implementation, disk and network I/O can be overlapped with CPU activity so that all computations are CPU bound.

As a practical optimization, we do not use the Hilbert class polynomial $H_\O$ in Step~1 of Algorithm~\ref{alg2}.  Instead, we compute the minimal polynomial of some more favorable class invariant, as described in~\cite{EngeSutherland:CRTClassInvariants}, which is then used to obtain a $j$-invariant.  Additionally, as noted in Section~\ref{selectorder}, it suffices to compute a class polynomial for the maximal order containing $\O$.  With these optimizations the time spent computing class polynomials is completely negligible (well under one second).

Another important optimization is the use of polynomial gcds to accelerate root-finding when walking paths in the isogeny graph, a technique developed in~\cite[\S 2]{EngeSutherland:CRTClassInvariants}.  This greatly accelerates the enumeration of the sets $\Ell_\O(\Fp)$ and $\Ell_R(\Fp)$ in Steps 3 and 5 of Algorithm~\ref{alg2}.  As a result, most of the computation (typically over 75\%) is spent interpolating polynomials in Steps 6 and 7.

\subsection{Computations over $\ZZ$}
Tables 1 and 2 provide performance data for computations of $\Phi_l$ and $\Phi_l^\mathfrak{f}$ using Algorithm~\ref{alg2}.  For each $l$ we list:
\begin{itemize}
\item The discriminant $D$ of the suitable order $\O$.
\item The number of CRT primes $n=\#S$ used.
\item The height bound $B_l$ in bits and the actual bit-size $b_l$ of the largest coefficient.
\item The total size of $\Phi_l$ (resp. $\Phi_l^\mathfrak{f}$) in megabytes (1MB = $10^6$ bytes), computed as the sum of the coefficient sizes, with symmetric terms counted only once.
\item The total CPU time, in seconds.  This includes the time to select $\O$.
\item The throughput, defined as the total size divided by the total CPU time.
\end{itemize}

\begin{table}
\begin{center}
\begin{tabular}{@{}rrrrrrrr@{}}
$l$ &$|D|$&$n$&$B_l$&$b_l$&size (MB)&time (s)&MB/s\\
\midrule
 101 &    216407 &  184 &   6511 &   5751 &   2.65 &   2.25 & 1.18\\
 211 &    393047 &  369 &  14949 &  13359 &   27.6 &   14.4 & 1.92\\
 307 &    837407 &  531 &  22748 &  20483 &   90.5 &   51.0 & 1.78\\
 401 &    626431 &  725 &  30640 &  27642 &    211 &    130 & 1.62\\
 503 &   3076175 &  870 &  39421 &  35686 &    431 &    264 & 1.63\\
 601 &    461351 & 1011 &  48027 &  43542 &    755 &    485 & 1.56\\
 701 &   1254871 & 1229 &  56953 &  51731 &   1227 &    863 & 1.42\\
 809 &    916599 & 1376 &  66731 &  60743 &   1926 &   1410 & 1.37\\
 907 &    986855 & 1517 &  75712 &  69017 &   2759 &   2010 & 1.37\\
1009 &   2871983 & 1728 &  85157 &  77653 &   3857 &   2910 & 1.32\\
2003 &  91696103 & 3410 & 180941 & 166095 &  33120 &  31800 & 1.04\\
3001 & 248329639 & 5122 & 281635 & 259272 & 117256 & 143000 & 0.82\\
4001 &  72135279 & 6939 & 385300 & 355707 & 287783 & 363000 & 0.79 \\
5003 &  67243191 & 8373 & 491355 & 454429 & 577740 & 749000 & 0.77 \\
\bottomrule
\end{tabular}
\\
\vspace{6pt}
\textsc{Table} 1.  Computations of $\Phi_l$ over $\Z$.\\
\end{center}
\end{table}

In the last column of Table 1 one can see the quasilinear performance of Algorithm~\ref{alg2} as a function of the size of $\Phi_l$, and the constant factors appear to be advantageous relative to other algorithms.  For example, computing $\Phi_{1009}$ with the evaluation/interpolation algorithm of~\cite{Enge:ModularPolynomials} uses approximately 100000 CPU seconds (scaled to our hardware platform), while Algorithm~\ref{alg2} needs less than 3000. 

\begin{table}
\begin{center}
\begin{tabular}{@{}rrrrrrrr@{}}
$l$ &$|D|$&$n$&$B_l$&$b_l$&size (MB)&time (s)&MB/s\\
\midrule
 1009 &     1391 &   33 &   1275 &   1099 &   2.34 &   1.59 & 1.47\\
 2003 &    37231 &   58 &   2542 &   2271 &   19.5 &   10.7 & 1.81\\
 3001 &    88879 &   88 &   3822 &   3611 &   69.6 &   47.7 & 1.46\\
 4001 &    53191 &  112 &   5201 &   4801 &    167 &    116 & 1.45\\
 5003 &    30959 &  136 &   6613 &   6228 &    339 &    241 & 1.41\\
 6007 &   463039 &  170 &   8052 &   7530 &    595 &    493 & 1.21\\
 7001 &   150631 &  192 &   9496 &   8876 &    957 &    701 & 1.37\\
 8009 &   315031 &  220 &  10979 &  10292 &   1453 &   1200 & 1.21\\
 9001 &   179159 &  240 &  12453 &  11974 &   2123 &   1790 & 1.18\\
10009 &   207919 &  265 &  13964 &  13453 &   2953 &   2630 & 1.12\\
20011 &  1114879 &  537 &  29485 &  27860 &  24942 &  27600 & 0.90\\
30011 &  2890639 &  795 &  45649 &  43304 &  87660 & 123000 & 0.71\\
40009 & 22309439 & 1032 &  62210 &  59439 & 214273 & 335000 & 0.64\\
50021 & 37016119 & 1316 &  79116 &  78077 & 508571 & 677000 & 0.75\\
60013 & 27334823 & 1594 &  96165 &  91733 & 747563 &1150000 & 0.65\\
\bottomrule
\end{tabular}
\\
\vspace{6pt}
\textsc{Table} 2.  Computations of $\Phi_l^{\mathfrak{f}}$ over $\Z$.\\
\vspace{2pt}
\end{center}
\end{table}

The first five rows of Table 2 may be compared to the corresponding rows of Table 1 to see the performance advantage gained when computing modular polynomials for the Weber $\mathfrak{f}$ function rather than $j$.  As expected, these polynomials are approximately 1728 times smaller, and the speedup achieved by Algorithm~\ref{alg2} is even better; we already achieve a speedup of around 1800 when $l=1009$, and this increases to to over 3000 when $l=5003$.  This can be explained by the superlinear complexity of interpolation, as well as the superior cache utilization achieved by condensing the sparse coefficients of $\Phi_l^\mathfrak{f}$, as described in Section~\ref{ComputingGamma2}.

As noted in Section~\ref{Weberf}, we used a heuristic height bound for the computations in Table~2.  The gap between the values of $b_l$ and $B_l$ in each case gives us high confidence in the results (the probability of this occurring by chance is negligible).

\subsection{Computations modulo $\boldsymbol{m}$}

Table~3 gives timings for computations of $\Phi_l$ modulo 256-bit and 1024-bit primes $m$.  The values of $m$ are arbitrary, and, in particular, they are not of a form suitable for direct computation with Algorithm~\ref{alg1}.  Instead, Algorithm~\ref{alg2} derives $\Phi_l\bmod m$ from the computations of $\Phi_l\bmod p$, for $p\in S$, using the explicit CRT.  The same set $S$ is used as when computing $\Phi_l$ over $\Z$, so the running time is largely independent of $m$, but using the explicit CRT yields a noticeable speedup when $\log m$ is significantly smaller than $6l\log l$.  For example, when $l=1009$ it takes approximately 2300 seconds to compute $\Phi_l\bmod m$, for the $m$ listed in Table~3, versus about 2900 seconds to compute $\Phi_l$ over $\Z$.

In addition to computing $\Phi_l\bmod m$ directly, we may also obtain $\Phi_l\bmod m$ by computing $\Phi_l^{\gamma_2}\bmod m$ and applying (\ref{GammaToj}), as discussed in Section~\ref{ComputingGamma2}.  The time to compute $\Phi_l^{\gamma_2}\bmod m$ is essentially independent of $m$, but the time to apply (\ref{GammaToj}) is not.  Even so, for the 256-bit and 1024-bit $m$ that we used, computing $\Phi_l\bmod m$ in this fashion is much faster than computing $\Phi_l\bmod m$ directly; for $l=1009$ we achieve times of 223 and 403 seconds, respectively.  As with $\Phi_l^{\mathfrak{f}}$, this speedup improves superlinearly, and for large~$l$ it exceeds the expected factor of 9.

When computing $\Phi_l^{\gamma_2}\bmod m$ we used the height bound $B_l^{\gamma_2}=2l\log l + 8l$ given by (\ref{gamma2height}).  The timings in Table~3 would be further improved if the heuristic bound $B_l^{\gamma_2}=2l\log l + 4l$ were used instead.

The computations listed in Tables 1 and 2 were practically limited by space, not time.  The largest computations took only a day or two when run on 16 cores, but required nearly a terabyte of disk storage.  However when computing $\Phi_l\bmod m$, we can handle larger values of $l$ without using an excessive amount of space.  When $l=20011$, for example, the total size of $\Phi_l$ is over 30 terabytes, but we are able to compute $\Phi_l$ modulo a 256-bit integer $m$ using less than 10 gigabytes.

\begin{table}
\begin{center}
\begin{tabular}{@{}rrrrrrrrrr@{}}
&&\multicolumn{3}{c}{$m = 2^{256}-189$}&&\multicolumn{3}{c}{$m = 2^{1024}-105$}\\
\cmidrule(r){3-5}\cmidrule(r){7-9}
$l$&&$\Phi_l$&$\Phi_l^{\gamma_2}$&$\Phi_l^*$&\hspace{12pt}&$\Phi_l$&$\Phi_l^{\gamma_2}$&$\Phi_l^*$\\
\midrule
  101 &&    2.12 &   0.16 &   0.47 &&    2.16 &   0.17 &   1.53 \\
  211 &&    12.4 &   1.64 &   3.26 &&    12.7 &   1.68 &   7.95 \\
  307 &&    43.3 &   4.82 &   8.34 &&    44.0 &   4.93 &   19.3 \\
  401 &&     109 &   10.9 &   17.9 &&     111 &   11.1 &   38.0 \\
  503 &&     215 &   23.3 &   34.0 &&     219 &   23.8 &   66.4 \\
  601 &&     390 &   40.5 &   55.8 &&     395 &   41.4 &    110 \\
  701 &&     695 &   69.1 &   90.2 &&     703 &   70.3 &    158 \\
  809 &&    1130 &    105 &    134 &&    1150 &    107 &    222 \\
  907 &&    1590 &    158 &    194 &&    1600 &    160 &    306 \\
 1009 &&    2300 &    223 &    267 &&    2320 &    225 &    403 \\
 2003 &&   23900 &   2400 &   2590 &&   24100 &   2440 &   3210 \\
 3001 &&  106000 &   9250 &   9650 &&  107000 &   9360 &  11200 \\
 4001 &&  283000 &  25100 &  25900 &&  287000 &  25400 &  28600 \\
 5003 &&  647000 &  57000 &  58300 &&  653000 &  60200 &  65700 \\
10009 && 7180000 & 681000 & 687000 && 7320000 & 688000 & 713000\\
\bottomrule
\end{tabular}
\\
\vspace{6pt}
\textsc{Table} 3.  Computations of $\Phi_l$ and $\Phi_l^{\gamma_2}$ modulo $m$.\\
\vspace{2pt}
\footnotesize
Columns $\Phi_l^*$ list the total time to obtain $\Phi_l$ by computing $\Phi_l^{\gamma_2}$ and applying (\ref{GammaToj}).
\normalsize
\end{center}
\end{table}

\section*{Acknowledgments}
We thank David Harvey for the {\tt zn\_poly} library, and Andreas Enge for providing timings for his evaluation/interpolation algorithm.
We also thank Igor Shparlinski for his helpful comments on an early draft of this paper.

\section*{Appendix}
\begin{lemma}\label{HWbound}
Let $c$ be a real number greater than $c_0=\log_2 e\approx 1.44$.
Let $\pi_c(x)$ count the integers $n\in[3,x]$ for which $\omega(n) \ge c\log\log n$.
Then $\pi_c(x) = O \bigl( x (\log x)^{1-c/c_0} \bigr).$
\end{lemma}
\begin{proof}
Let $d(n)$ count the divisors of $n$.  From \cite[Thm.~320]{Hardy:NumberTheory} we have
\[
\sum_{n\le x}2^{\omega(n)}\le \sum_{n\le x}d(n) = x\log x + O(x).
\]
At most $O\bigl(x(\log x)^{1-c/c_0}\bigr)$ terms on the LHS have $n\ge \sqrt{x}$ and $\omega(n)\ge c\log\log x$.
Applying $\sqrt{x}=O\bigl(x(\log x)^{1-c/c_0}\bigr)$ and $\frac{\log\log\sqrt{x}}{\log\log x} = 1+o(1)$ yields the lemma.
\end{proof}

\bibliographystyle{amsplain}

\begin{thebibliography}{10}

\bibitem{Bach:ERHbounds}
Eric Bach, \emph{Explicit bounds for primality testing and related problems},
  Mathematics of Computation \textbf{55} (1990), no.~191, 355--380.

\bibitem{Belding:HilbertClassPolynomial}
Juliana Belding, Reinier Br\"{o}ker, Andreas Enge, and Kristin Lauter,
  \emph{Computing {H}ilbert class polynomials}, Algorithmic Number Theory
  Symposium--{ANTS VIII} (A.~J. {van der Poorten} and A.~Stein, eds.), Lecture
  Notes in Computer Science, vol. 5011, Springer, 2008, pp.~282--295.

\bibitem{Berlekamp:PolyFactoringLargeFF}
Elwyn~R. Berlekamp, \emph{Factoring polynomials over large finite fields},
  Mathematics of Computation \textbf{24} (1970), no.~111, 713--735.

\bibitem{Bernstein:ModularExponentiation}
Daniel~J. Bernstein, \emph{Modular exponentiation via the explicit {C}hinese
  {R}emainder {T}heorem}, Mathematics of Computation \textbf{76} (2007),
  443--454.

\bibitem{Biehl:FormReductionComplexity}
Ingrid Biehl and Johannes Buchmann, \emph{An analysis of the reduction
  algorithms for binary quadratic forms}, Voronoi's Impact on Modern Science
  (P.~Engel and H.~Syta, eds.), Institute of Mathematics, Kyiv, 1998, available
  at \url{http://www.cdc.informatik.tu-darmstadt.de/reports/TR/TI-97-26.ps.gz},
  pp.~71--98.

\bibitem{BissonSutherland:Endomorphism}
Gaetan Bisson and Andrew~V. Sutherland, \emph{Computing the endomorphism ring
  of an ordinary elliptic curve over a finite field}, Journal of Number Theory
  (2009), to appear, \url{http://arxiv.org/abs/0902.4670}.

\bibitem{Blake:EllipticCurves}
Ian Blake, Gadiel Seroussi, and Nigel Smart, \emph{Elliptic curves in
  cryptography}, London Mathematical Society Lecture Note Series, vol. 265,
  Cambridge University Press, 1999.

\bibitem{Blake:ModularPolynomials}
Ian~F. Blake, J\'{a}nos~A. Csirik, Michael Rubinstein, and Gadiel Seroussi,
  \emph{On the computation of modular polynomials for elliptic curves}, Tech.
  report, Hewlett-Packard Laboratories, 1999,
  \url{http://www.math.uwaterloo.ca/~mrubinst/publications/publications.html}.

\bibitem{Bostan:FastIsogenies}
Alin Bostan, Bruno Salvy, Fran\c{c}ois Morain, and \'{E}ric Schost, \emph{Fast
  algorithms for computing isogenies between elliptic curves}, Mathematics of
  Computation \textbf{77} (2008), 1755--1778.

\bibitem{Broker:Thesis}
Reinier Br\"{o}ker, \emph{Constructing elliptic curves of prescribed order},
  {P}h{D} thesis, Universiteit Leiden, 2006.

\bibitem{Broker:pAdicClassPolynomial}
\bysame, \emph{A $p$-adic algorithm to compute the {H}ilbert class polynomial},
  Mathematics of Computation \textbf{77} (2008), 2417--2435.

\bibitem{Broker:pAdicClassInvariants}
\bysame, \emph{$p$-adic class invariants}, LMS Journal of Computation and
  Mathematics (2010), to appear.

\bibitem{BrokerSutherland:PhiHeightBound}
Reinier Br{\"o}ker and Andrew~V. Sutherland, \emph{An explicit height bound for
  the classical modular polynomial}, Ramanujan Journal \textbf{22} (2010),
  293--313.

\bibitem{Buchmann:BinaryQuadraticForms}
Johannes Buchmann and Ulrich Vollmer, \emph{Binary quadratic forms: an
  algorithmic approach}, Algorithms and Computations in Mathematics, vol.~20,
  Springer, 2007.

\bibitem{Magma}
J.J. Cannon and W.~Bosma (Eds.), \emph{Handbook of {Magma} functions}, 2.15
  ed., 2008, available at
  \url{http://magma.maths.usyd.edu.au/magma/htmlhelp/MAGMA.htm}.

\bibitem{Castagnos:NICECryptanalysis}
Guilhem Castagnos and Fabien Laguillaumie, \emph{On the security of
  cryptosystems with quadratic decryption: the nicest cryptanalysis}, Advances
  in Cryptology: EUROCRYPT 2009 (A.~Joux, ed.), Lecture Notes in Computer
  Science, vol. 5479, Springer, 2009, pp.~260--277.

\bibitem{CharlesLauter:ModPoly}
Denis Charles and Kristin Lauter, \emph{Computing modular polynomials}, LMS
  Journal of Computation and Mathematics \textbf{8} (2005), 195--204.

\bibitem{Cohen:CANT2}
Henri Cohen, \emph{Advanced topics in computational number theory}, Springer,
  2000.

\bibitem{Cohen:HECHECC}
Henri Cohen and Gerhard~Frey et~al., \emph{Handbook of elliptic and
  hyperelliptic curve cryptography}, Chapman and Hall, 2006.

\bibitem{CohenPaula:ModularPolynomials}
Paula Cohen, \emph{On the coefficients of the transformation polynomials for
  the elliptic modular function}, Math. Proc. of the Cambridge Philosophical
  Society \textbf{95} (1984), 389--402.

\bibitem{Cox:ComplexMultiplication}
David~A. Cox, \emph{Primes of the form $x^2+ny^2$: {F}ermat, class field
  theory, and complex multiplication}, John Wiley and Sons, 1989.

\bibitem{Elkies:AtkinBirthday}
Noam~D. Elkies, \emph{Elliptic and modular curves over finite fields and
  related computational issues}, Computational Perspectives on Number Theory
  (D.~A. Buell and J.~T. Teitelbaum, eds.), Studies in Advanced Mathematics,
  vol.~7, AMS, 1998, pp.~21--76.

\bibitem{Enge:FloatingPoint}
Andreas Enge, \emph{The complexity of class polynomial computation via floating
  point approximations}, Mathematics of Computation \textbf{78} (2009),
  1089--1107.

\bibitem{Enge:ModularPolynomials}
\bysame, \emph{Computing modular polynomials in quasi-linear time}, Mathematics
  of Computation \textbf{78} (2009), 1809--1824.

\bibitem{Enge:GeneralizedWeberI}
Andreas Enge and Francois Morain, \emph{Generalized {W}eber functions {I}},
  2009, \url{http://arxiv.org/abs/0905.3250}.

\bibitem{Enge:DoubleEtaQuotient}
Andreas Enge and Reinhard Schertz, \emph{Constructing elliptic curves over
  finite fields using double eta-quotients}, Journal de Th\'eorie des Nombres
  de Bordeaux \textbf{16} (2004), no.~3, 555--568.

\bibitem{EngeSchertz:CompositeLevel}
\bysame, \emph{Modular curves of composite level}, Acta Arithmetica
  \textbf{118} (2005), no.~2, 129--141.

\bibitem{EngeSutherland:CRTClassInvariants}
Andreas Enge and Andrew~V. Sutherland, \emph{Class invariants for the {CRT}
  method}, Algorithmic Number Theory Symposium--{ANTS IX} (G.~Hanrot,
  F.~Morain, and E.~Thom\'{e}, eds.), Lecture Notes in Computer Science, vol.
  6197, Springer-Verlag, 2010, pp.~142--156.

\bibitem{GNU}
Free~Software Foundation, \emph{{GNU} compiler collection}, January 2010,
  version 4.4.3, available at \url{http://gcc.gnu.org/}.

\bibitem{Fouquet:IsogenyVolcanoes}
Mireille Fouquet and Fran\c{c}ois Morain, \emph{Isogeny volcanoes and the {SEA}
  algorithm}, Algorithmic Number Theory Symposium--{ANTS V} (C.~Fieker and
  D.~R. Kohel, eds.), Lecture Notes in Computer Science, vol. 2369, Springer,
  2002, pp.~276--291.

\bibitem{Galbraith:GHSattack}
Steven~D. Galbraith, Florian Hess, and Nigel~P. Smart, \emph{Extending the
  {GHS} {W}eil descent attack}, Advances in Cryptology---{EUROCRYPT 2002},
  Lecture Notes in Computer Science, vol. 2332, Springer, 2002, pp.~29--44.

\bibitem{Gee:GeneratingClassFields}
Alice Gee and Peter Stevenhagen, \emph{Generating class fields with {S}himura
  reciprocity}, Algorithmic Number Theory Symposium--{ANTS III}, Lecture Notes
  in Computer Science, vol. 1423, Springer, 1998, pp.~442--453.

\bibitem{GMP}
Torbj\"{o}rn {Granlund et al.}, \emph{{GNU} multiple precision arithmetic
  library}, September 2010, version 5.0.1, available at
  \url{http://gmplib.org/}.

\bibitem{Hardy:NumberTheory}
Godfrey~H. Hardy and Edward~M. Wright, \emph{An introduction to the theory of
  numbers}, fifth ed., Oxford Science Publications, 1979.

\bibitem{Harvey:zn_poly}
David Harvey, \emph{$\text{zn\_poly}$: a library for polynomial arithmetic},
  2008, version 0.9, \url{http://cims.nyu.edu/~harvey/zn_poly}.

\bibitem{Harvey:KroneckerSubstitution}
\bysame, \emph{Faster polynomial multiplication via multipoint {K}ronecker
  substitution}, Journal of Symbolic Computation \textbf{44} (2009), no.~10,
  1502--1510.

\bibitem{Herrmann:FourierCoefficients}
Oskar Herrmann, \emph{Uber die {B}erechnung der {F}ourierkoeffizienten der
  {F}unktion $j(\tau)$}, J. Reine Agnew. Math. \textbf{274/275} (1975),
  187--195.

\bibitem{Holt:CGTHandbook}
Derek~F. Holt, Bettina Eick, and Eamonn~A. O'Brien, \emph{Handbook of
  computational group theory}, CRC Press, 2005.

\bibitem{Ito:ModularEquation}
Hideji Ito, \emph{Computation of the modular equation}, Proc. Japan Acad. Ser.
  A \textbf{71} (1995), 48--50.

\bibitem{Jarden:Density}
Moshe Jarden, \emph{Transfer principles for finite and p-adic fields}, Nieuw
  Archief voor Wiskunde \textbf{3} (1980), no.~28, 139--158,
  \url{http://www.tau.ac.il/~jarden/Articles/paper27.pdf}.

\bibitem{Kaltofen:ModularEquation}
Erich Kaltofen and Noriko Yui, \emph{On the modular equation of order $11$},
  Proceedings of the 1984 MACSYMA Users Conference, 1984, pp.~472--485.

\bibitem{Kohel:thesis}
David Kohel, \emph{Endomorphism rings of elliptic curves over finite fields},
  {PhD} thesis, University of California at Berkeley, 1996.

\bibitem{Lagarias:Chebotarev}
J.~C. Lagarias and A.~M. Odlyzko, \emph{Effective versions of the {C}hebotarev
  density theorem}, Algebraic number fields: {$L$}-functions and {G}alois
  properties (Proc. Sympos., Univ. Durham, Duram, 1975), Academic Press, 1977,
  pp.~409--464.

\bibitem{Lang:EllipticFunctions}
Serge Lang, \emph{Elliptic functions}, second ed., Springer-Verlag, 1987.

\bibitem{LMMS:PointCounting}
Frank Lehmann, Markus Maurer, Volker M\"{u}ller, and Victor Shoup,
  \emph{Counting the number of points on elliptic curves over finite fields of
  characteristic greater than three}, Algorithmic Number Theory
  Symposium--{ANTS I} (L.~M. Adleman and M.-D. Huang, eds.), Lecture Notes in
  Computer Science, vol. 877, 1994, pp.~60--70.

\bibitem{Lenstra:RigorousFactoring}
Hendrik~W. {Lenstra, Jr.} and Carl Pomerance, \emph{A rigorous time bound for
  factoring integers}, Journal of the American Mathematical Society \textbf{5}
  (1992), no.~3, 483--516.

\bibitem{Morain:PointCounting}
Fran\c{c}ois Morain, \emph{Calcul du nombre de points sur une courbe elliptique
  dans un corps fini: aspects algorithmiques}, Journal de Th\'{e}orie des
  Nombres de Bordeaux \textbf{7} (1995), no.~1, 111--138.

\bibitem{Muller:thesis}
Volker M\"uller, \emph{Ein {A}lgorithmus zur {B}estimmung der {P}unktanzahl
  elliptischer {K}urven \"uber endlichen {K}\"orpern der {C}harakteristik
  gr\"o\ss er drei}, {PhD} thesis, Universit\"at des Saarlandes, 1995.

\bibitem{Neukirch:AlgebraicNumberTheory}
J\"{u}rgen Neukirch, \emph{Algebraic number theory}, Springer, 1999.

\bibitem{Schonhage:Multiplication}
Arnold Sch\"{o}nhage and Volker Strassen, \emph{Schnelle {M}ultiplikation
  gro\ss{}er {Z}ahlen}, Computing \textbf{7} (1971), 281--292.

\bibitem{Schoof:ECPointCounting2}
Ren\'{e} Schoof, \emph{Counting points on elliptic curves over finite fields},
  Journal de Th\'{e}orie des Nombres de Bordeaux \textbf{7} (1995), 219--254.

\bibitem{SAGE}
William Stein and David Joyner, \emph{{SAGE}: System for {A}lgebra and
  {G}eometry {E}xperimentation}, Communications in Computer Algebra (SIGSAM
  Bulletin) (2005), 61--64.

\bibitem{Stevenhagen:NumberRings}
Peter Stevenhagen, \emph{The arithmetic of number rings}, Algorithmic Number
  Theory: Lattices, Number Fields, and Cryptography (J.P. Buhler and
  P.~Stevenhagen, eds.), Mathematical Sciences Research Institute Publications,
  vol.~44, Cambridge University Press, 2008.

\bibitem{Sutherland:HilbertClassPolynomials}
Andrew~V. Sutherland, \emph{Computing {H}ilbert class polynomials with the
  {C}hinese {R}emainder {T}heorem}, Mathematics of Computation \textbf{80}
  (2011), 501--538.

\bibitem{Velu:Isogenies}
Jacques V\'{e}lu, \emph{Isog\'{e}nies entre courbes elliptiques}, Comptes
  Rendus Hebdomadaires des S\'{e}ances de l'Acad\'{e}mie des Sciences,
  S\'{e}ries A et B \textbf{273} (1971), 238--241.

\bibitem{Gathen:ComputerAlgebra}
Joachim von~zur Gathen and J\"{u}rgen Gerhard, \emph{Modern computer algebra},
  second ed., Cambridge University Press, 2003.

\bibitem{Washington:EllipticCurves}
Lawrence~C. Washington, \emph{Elliptic curves: Number theory and cryptography},
  second ed., CRC Press, 2008.

\bibitem{Weber:Algebra}
Heinrich Weber, \emph{Lehrbuch der algebra}, third ed., vol. III, Chelsea,
  1961.

\end{thebibliography}
\providecommand{\bysame}{\leavevmode\hbox to3em{\hrulefill}\thinspace}
\providecommand{\MR}{\relax\ifhmode\unskip\space\fi MR }
\providecommand{\MRhref}[2]{%
  \href{http://www.ams.org/mathscinet-getitem?mr=#1}{#2}
}
\providecommand{\href}[2]{#2}

\end{document}